\documentclass[12pt]{amsart}
\usepackage{amsmath,amssymb}    
\usepackage{url}                
\usepackage{graphicx}           
\usepackage{amsthm}
\usepackage{fullpage}
\theoremstyle{definition}
\numberwithin{equation}{section}
\newtheorem{theorem}[equation]{Theorem}
\newtheorem{lemma}[equation]{Lemma}
\newtheorem{proposition}[equation]{Proposition}
\newtheorem{corollary}[equation]{Corollary}

\newtheorem{definition}[equation]{Definition}
\newtheorem{example}[equation]{Example}

\newtheorem{remark}[equation]{Remark}

\date{}

\begin{document}

\title{Galois bimodules and integrality 
of PI comodule algebras over invariants}

\author{Pavel Etingof}
\address{Department of Mathematics, Massachusetts Institute of Technology,
Cambridge, MA 02139, USA}
\email{etingof@math.mit.edu}

\begin{abstract}
Let $A$ be a comodule algebra for a finite dimensional Hopf
algebra $K$ over an algebraically closed field $k$, 
and let $A^K$ be the subalgebra of invariants. 
Let $Z$ be a central subalgebra 
in $A$, which is a domain with quotient field $Q$.
Assume that $Q\otimes_Z A$ is a central simple algebra over $Q$, 
and either $A$ is a finitely generated torsion-free $Z$-module and 
$Z$ is integrally closed in $Q$, or $A$ is a finite projective
$Z$-module. Then we show that $A$ and $Z$ are 
integral over the subring of central invariants $Z\cap A^K$. 
More generally, we show that this 
statement is valid under the same assumptions 
if $Z$ is a reduced algebra with quotient ring $Q$, 
and $Q\otimes_Z A$ is a semisimple algebra with center $Q$. 
In particular, the statement holds 
for a coaction of $K$ on a prime PI algebra $A$ 
whose center $Z$ is an integrally closed finitely generated 
domain over $k$. For the proof, we develop a theory of Galois bimodules
over semisimple algebras finite over the center. 
\end{abstract}

\maketitle


\section{Introduction}

The goal of this paper is to prove a noncommutative analogue (for
PI algebras) of Skryabin's
integrality theorem for coactions of a finite dimensional Hopf
algebra $K$ on a commutative algebra $A$. In \cite{S1},
Theorem 2.5 and Proposition 2.7, Skryabin showed under 
a very minor assumption (that $A$ has no nonzero 
nilpotent $K$-costable ideals, which is not needed in positive
characteristic) that such an algebra $A$ is integral over its
invariants $A^K$. This is a generalization of the classical theorem 
of E. Noether saying that a commutative algebra is integral 
over its invariants under a finite group action. 

We extend Skryabin's result to the case when $A$ is
a noncommutative PI algebra over an algebraically closed field $k$, 
under some assumptions. Namely, let $Z\subset A$ be a central subalgebra 
which is a domain, such that $A$ is a finitely
generated $Z$-module. We show in Theorem 3.1 that if 
(I) $Q\otimes_Z A$ is a
central simple algebra over $Q$ (where $Q$ is the quotient field of $Z$), 
and (II) $A$ is a projective $Z$-module, 
or $A$ is a torsion free $Z$-module and $Z$ is integrally closed in
$Q$, then $Z$ and hence $A$ is integral over $Z\cap A^K$. 
In particular, we show that this property holds 
for a coaction of $K$ on a prime PI algebra $A$ 
whose center $Z$ is an integrally closed finitely generated 
domain over $k$.
\footnote{In the case when $K$ is basic (i.e., 
$K^*$ is pointed), similar (but not equivalent) 
results were obtained by A. Totok in
\cite{T}, Theorem 2.5. Also, some results about integrality and finiteness 
of noncommutative module algebras over invariants are obtained in \cite{E}.} 
Further, we show that 
the degree of $Q$ over $Q^K$ (the field of quoteints of
$Z\cap A^K$) divides the dimension of $K$ (Proposition
\ref{divisi}).   

We also generalize these results to the case when $Z$ is not necessarily a domain. 
Namely, we show in Theorem 
\ref{maint1} that the same integrality result holds 
if $Z$ is a reduced algebra, and $Q\otimes_Z A$ 
is a semisimple algebra with center $Q$. Finally, in Proposition \ref{divisi2}
we generalize the divisibility result of Proposition \ref{divisi}. 

Skryabin's proof is based on the freeness theorem 
for finite dimensional Hopf algebras over coideal subalgebras, 
which he uses to show that $A$ has invariant characteristic
polynomials. Our proof is ideologically similar to Skryabin's 
proof. Namely, we show that the coefficients of 
the minimal polynomial and the 
characteristic polynomial
of the operator of right multiplication by $z\in Z$
on $Q\otimes_Z A$ belong to $Z$ and are $K$-invariant. 
However, in the noncommutative 
setting the coaction does not give rise to coideal subalgebras, 
so we use a
different method, based on minimal polynomials.
\footnote{We expect, however, that by using the results of 
\cite{S2} on projectivity and freeness over comodule algebras, 
one can obtain different proofs of our results, which 
are more direct generalizations of the arguments of \cite{S1}.}

Namely, we develop a theory of {\it Galois  
bimodules} over fields and, more
generally, over central simple algebras and 
semisimple algebras finite over the center, 
which may be of independent interest. 
For an algebra $A$, a Galois $A$-bimodule of rank $d$ 
is an $A$-bimodule $P$ which is free of rank $d$ on each side,
and satisfies the equation $P\otimes_A P\cong P^d$. 
We prove a classification theorem for Galois bimodules over a
field $L$ (Theorem \ref{classif}) 
which says that such a bimodule is simply a multiple of
$L\otimes_F L$, where $F$ is the center of $P$, i.e. the largest
subalgebra of $L$ over which $P$ is linear. Similarly, we show that
any Galois bimodule over a central simple algebra $B$ over a
field $L$ is a rational multiple of $B\otimes_F B$ (Proposition \ref{censim1}).
The same statement holds more generally (for connected Galois bimodules) 
if $B$ is a semisimple algebra which is finite 
over its center (Proposition \ref{censim2}). 
\footnote{We also introduce the notion of a weakly Galois
  bimodule, replacing the condition that $P\otimes_A P=P^d$ by 
a weaker condition that $P\otimes_A P$ is contained in $P^N$ for
some $N$, and show that much of the theory of Galois bimodules in
fact goes through for weakly Galois bimodules. Finally, we 
study quasi-Galois bimodules over commutative semisimple algebras $L$, 
i.e. bimodules $P$ finite on both sides and such that $P\otimes_L P\cong P^d$, and provide their classification (Theorem \ref{commsem}).}  

The theory of Galois bimodules is applied 
to the problem of invariants in the following way. We show 
that the Hopf coaction of $K$ on $A$ makes the tensor product
\footnote{Throughout the paper, $\otimes$ without subscript will denote the tensor product over the ground field $k$.} 
$Q\otimes_Z A\otimes K$ into a Galois bimodule over the algebra $Q\otimes_Z A$. 
This allows us to establish
the invariance of the characteristic polynomials 
of the right action of $z\in Z$ on $Q\otimes_Z P$ and 
thus obtain the desired integrality result (Theorem \ref{maint}, Theorem \ref{maint1}). 
\footnote{When we pass from $A$ to the associated Galois bimodule, 
we forget almost everything about the coaction, 
but remember just enough to study the invariants of the coaction 
in the center of $A$.} 

We also relate Galois bimodules to  
tensor subcategories of the tensor category 
of bimodules over a field 
studied in \cite{FO} and of the
tensor category of bimodules over a central simple algebra, 
and explain that these subcategories 
are twisted forms of categories 
of bimodules over semisimple algebras in the sense of 
\cite{EG}. 

Finally, our approach motivates the definition of a new invariant 
of a Hopf coaction on a central simple algebra 
- its Galois group. It would be interesting to
study this invariant further. 

The paper is organized as follows. 
In Section 2 we develop the theory of Galois and weakly Galois bimodules 
over fields and central simple algebras, and 
obtain the classification of Galois bimodules. 
In Section 3, we apply the theory of Section 2 
to prove the integrality theorem (Theorem \ref{maint})
in the case when $Z$ is a domain, and 
discuss various consequences and examples. 
In Section 4 we develop the theory of Galois bimodules over 
semisimple algebras finite over the center, and obtain their 
classification. In Section 5 we apply the results of Section 4 to obtain integrality 
(Theorem \ref{maint1}) in the case when $Z$ is a reduced algebra.
Finally, in Section 6 we discuss the connection of the theory of Galois bimodules for fields 
with tensor subcategories of the category of bimodules over fields discussed in \cite{FO}. 

{\bf Acknowledgements.} 
This work was partially supported
by the NSF grant DMS-1000113. The 
author is very grateful to Chelsea Walton
for introducing him to this field,
for many useful discussions 
and ideas that inspired this work,
and for helpful comments on this paper. 

\section{Galois bimodules}

\subsection{Bimodules} 

Let $P$ be a bimodule over a unital ring $A$. 
Then we have a homomorphism  
$\phi_P: A^{op}\to {\rm End}_A(P)$ 
induced by the right action of $A$. 
We will denote $\phi_P$
just by $\phi$, when no confusion is possible. 

Assume from now on that $P$ is free of
rank $d\ge 1$ as a left module. 
Then, fixing an identification 
$P\cong A^d$ as left modules, 
we can view $\phi_P$ as a homomorphism 
$A\to {\rm Mat}_d(A)$. 

\begin{definition} 
We say that $P$  
is {\it linear} over a central subring $k\subset A$ if the left and right action 
of $k$ coincide on $P$. The {\it center} of $P$ is the set ${\mathcal Z}(P)$ 
of $z\in A$ such that $zx=xz$ for all $x\in P$, i.e. 
such that $\phi_P(z)=z\cdot {\rm Id}$.  
\end{definition}

\begin{remark} 1. An 
$A$-bimodule linear over $k$ is the same thing as a left module 
over $A\otimes_k A^{\rm op}$. 

2. ${\mathcal Z}(P)\subset A$ 
is the maximal central subring 
of $A$ over which $P$ is linear. 
Indeed, if $z\in {\mathcal Z}(P)$ 
then for any $x\in P$, $a\in A$,
$(az)x=a(zx)=a(xz)=(ax)z=z(ax)=(za)x$, hence 
$za=az$ since $P$ is free over $A$. Thus $z$ belongs to the
center of $A$.  
\end{remark}

Let $P$ be a bimodule over a field $L$ 
which has dimension $d$ as a left and right $L$-vector
space. 

\begin{definition} We say that $P$ is {\it split} if for each $a\in L$, 
the eigenvalues of $\phi_P(a)$ 
are contained in $L$. In this case, we say that $P$ is {\it split
separable} if in addition the matrices $\phi_P(a)$ are
diagonalizable for all $a$.
\end{definition} 

Given an algebra $A$ and an automorphism $g\in {\rm Aut}(A)$, 
let $Ag$ denote the $A$-bimodule which is $A$ as a left $A$-module,
while the right $A$-action is multiplication twisted by $g$:
$a\circ x\circ b=axg(b)$ for $a,x,b\in A$.

\begin{lemma}\label{fil} An $L$-bimodule $P$ is split if and only if it 
has a finite filtration whose successive quotients are of the form
$Lg$, where $g\in {\rm Aut}(L)$. Moreover, a split $L$-bimodule
$P$ is separable if and only if 
it is semisimple as a bimodule. 
\end{lemma}

\begin{proof} Assume $P$ is split. 
Then there exists a basis $v_1,...,v_d$ of $P$ as a left vector space 
in which $\phi_P(a)$ are upper triangular for all $a\in L$. 
Let ${\bold F}_iP$ be the span of $v_1,...,v_i$. 
Then ${\bold F}_\bullet$ is a bimodule filtration,
and ${\bold F}_{i+1}P/{\bold F}_iP$ is $L$ with 
the usual left action of $L$ and right action given by some field embedding 
$g_i: L\to L$. Since $P$ has dimension $d$ as a right vector space, these embeddings are all 
isomorphisms. Thus, ${\bold F}_{i+1}P/{\bold F}_iP=Lg_i$ for some $g_i\in {\rm Aut}(L)$.
Conversely, it is clear that any bimodule having a filtration with successive quotients $Lg$ is necessarily split.

Now, if $P$ is separable, then the basis $\lbrace{ v_i\rbrace}$ can be chosen in such a way that 
the matrices $\phi_P(a)$ are diagonal, so we get that $P=\oplus_i Lg_i$, i.e. $P$ is semisimple. 
Conversely, if $P$ is semisimple then the filtration ${\bold F}_\bullet$ must split, which implies that 
$P$ is separable. The lemma is proved.    
\end{proof} 

\subsection{Galois and weakly Galois bimodules}  

\begin{definition} A {\it weakly Galois bimodule} over $A$ of rank $d\ge 1$ 
is an $A$-bimodule $P$ such that $P\cong A^d$ as a left and right $A$-module, 
and $P\otimes_A P$ is contained in $P^N$ as an $A$-bimodule for some $N$. 
If moreover $P\otimes_A P\cong P^d$ as an $A$-bimodule, then we call $P$ 
a {\it Galois bimodule} of rank $d$. 
\end{definition}

\begin{example}\label{exaa}
1. $A^d$ is a Galois $A$-bimodule of rank $d$. 

2. If every right-invertible element of $A$ is left-invertible, 
and if $P$ is a Galois $A$-bimodule of rank $1$ 
then $P=A$. Indeed, $\phi: A\to A$ is a homomorphism
which makes $A$ into a free right module of rank $1$ over itself. 
This means that $\phi$ is injective. Fix a generator $x\in A$ of this module. 
Then for any $y\in A$ there exists $a_y\in A$ such that 
$y=x\phi(a_y)$. In particular, $1=x\phi(a_1)$, hence 
$x$ is right invertible and thus invertible. 
So the equality $xy=x\phi(a_{xy})$ implies 
$y=\phi(a_{xy})$, i.e. $\phi$ is surjective, and hence an 
automorphism. Moreover, since $P$ is a Galois bimodule, 
$\phi^2=\phi$ modulo inner automorphisms, hence $\phi$ is inner and we can assume that 
$\phi={\rm Id}$, as claimed.  

3. If $A$ is a commutative algebra,
and $B$ is an $A$-algebra which is free of rank $d$ as an
$A$-module (e.g. if $B$ is a field extension of degree $d$ of a
field $A$) then $B\otimes_A B$ is a Galois bimodule over $B$ of
rank $d$. 

4. Let $G\subset {\rm Aut}(A)$ be a finite subgroup, let $n_g>0$
be integers, and let $P=\oplus_{g\in G}(Ag)^{n_g}$. 
Then $P$ is a weakly Galois $A$-bimodule of rank $d=\sum_{g\in G}n_g$.
Moreover, if all $n_g$ are equal, then $P$ is a Galois $A$-bimodule.   
\end{example}

\begin{lemma}\label{ext}
(i) Suppose that $B$ is an extension of $A$ 
such that $B$ is isomorphic to $A^n$ as an $A$-bimodule 
(for instance, if $A\subset B$ is a field extension of degree
$n$). In this case, if $P$ is a Galois (respectively, weakly Galois) bimodule 
of rank $d$ over $A$, then 
$B\otimes_A P\otimes_A B$ is a Galois (respectively, weakly Galois) 
bimodule of rank $dn$ over $B$. 

(ii) If $P$ is a Galois (respectively, weakly Galois)
bimodule over $A$ of rank $d$ linear over $k$, and 
$B$ is another $k$-algebra, then 
$P\otimes_k B$ is a Galois (respectively, weakly Galois)
bimodule over $A\otimes_k B$ 
of rank $d$. For instance, for $B={\rm Mat}_m(k)$, we have 
that $P\otimes_k B={\rm Mat}_m(P)$ is a (weakly) Galois bimodule over $B$.  

(iii) If $B$ is a central simple algebra with center $A$ 
and $\dim_A B=m^2$, then the restriction to $A$ 
of any Galois (respectively, weakly Galois) $B$-bimodule $P$ of rank $d$ 
is a Galois (respectively, weakly Galois) bimodule of rank $dm^2$. 
\end{lemma}

\begin{proof}
(i) Let $P$ be a Galois (respectively, weakly Galois) $A$-bimodule of rank $d$. 
Then it is clear that under the assumptions of (i), 
$B\otimes_A P\otimes_A B$ is free as a left and right $B$-module of rank $dn$.
Also 
$$
B\otimes_A P\otimes_A B\otimes_B B\otimes_A P\otimes_A B
=(B\otimes_A P\otimes_B P\otimes_A B)^n,
$$
which equals 
$(B\otimes_A P\otimes_A B)^{dn}$ (respectively, is contained 
in $(B\otimes_A P\otimes_A B)^{Nn}$ for some $N$),   
as desired. 

(ii) This follows from the fact that 
$(P\otimes_k B)\otimes_{A\otimes_k B}(P\otimes_k B)=
(P\otimes_A P)_{\otimes_k} B$. 

(iii) We have $B\otimes_A B^{op}\cong {\rm End}_A(B)$, so 
any $B$-bimodule linear over $A$ is a multiple of $B$.
In particular, $B\otimes_A B=B^{m^2}$ as a $B$-bimodule. 
Thus, $P\otimes_A P=P\otimes_B B\otimes_A B\otimes_B P=
(P\otimes_B P)^{m^2}$, which implies the statement.  
\end{proof}

\subsection{Galois bimodules coming from Hopf coactions}

Let $K$ be a finite dimensional Hopf algebra over an
algebraically closed field $k$ with coproduct $\Delta$, counit $\varepsilon$, 
and antipode $S$. Let $A$ be a (unital) $k$-algebra. 
Recall (\cite{M}) that $A$ is called 
a {\it right $K$-comodule algebra} if it is equipped with a 
coaction map $\rho: A\to A\otimes K$ 
which is an algebra homomorphism, and such that 
it equips $A$ with a structure of a $K$-comodule, i.e., 
$$
(\rho\otimes 1)(\rho(a))=(1\otimes \Delta)(\rho(a)),\
(1\otimes \varepsilon)(\rho(a))=a,\ a\in A.
$$

Now suppose that $A$ is an algebra 
with a coaction of a finite dimensional Hopf algebra $K$
of dimension $d$. So we have a coaction map 
$\rho: A\to A\otimes K$. 
This map equips the tensor product $P:=A\otimes K$ with a 
structure of an $A$-bimodule,
via $a\circ x\circ b:=(a\otimes 1)x\rho(b)$, $a,b\in A$, $x\in P$. 

\begin{proposition}\label{coa}
$P$ is a Galois bimodule over $A$ of rank $d$.
Moreover, one has ${\mathcal Z}(P)={\mathcal Z}(A)^K$, where
${\mathcal Z}(A)^K:={\mathcal Z}(A)\cap A^K$ is the subalgebra of
central invariants of $A$. 
\end{proposition} 

\begin{proof}
This is well known, but we will give a proof for reader's convenience. 

It is clear that $P\cong A^d$ as a left $A$-module. 
Let us show that this is also the case as a right module. 
To this end, define the linear map 
$\psi: A\otimes K\to A\otimes K$ by $\psi(a\otimes y)=(1\otimes y)\rho(a)$. 
Then 
$$
\psi(a_1\otimes y)\rho(a_2)=(1\otimes y)\rho(a_1)\rho(a_2)=(1\otimes y)\rho(a_1a_2)=\psi(a_1a_2\otimes y),
$$
which shows that $\psi$ is a homomorphism of right modules, where the target module is just the tensor product of $A$ 
with the vector space $K$. So it is enough to show that $\psi$ is an isomorphism. To do so, define another map 
$\xi: A\otimes K\to A\otimes K$ by 
$$
\xi(a\otimes y)=(1\otimes y)(1\otimes S^{-1})(\rho(a)).
$$
Then it is easy to check that $\xi$ is the inverse of $\psi$, which implies that $\psi$ is an isomorphism. 

Finally, let us show that $P\otimes_A P$ is isomorphic to $P^d$. Actually, the coaction gives rise to an explicit isomorphism. 
Consider the linear map $\tau: P\otimes_A P\to P\otimes K=A\otimes K\otimes K$ defined by the formula 
$$
\tau(a_1\otimes y_1\otimes a_2\otimes y_2):=(a_1\otimes y_1)\rho(a_2)\otimes y_2.
$$
It is easy to see that this map is well defined, is an isomorphism, and 
commutes with the left action of $A$ in the first component. 
Identifying $P\otimes_A P$ with $A\otimes K\otimes K$ using $\tau$, 
we can transport the $A$-bimodule structure of $P\otimes_A P$ to $A\otimes K\otimes K$,
and the transported structure is given by 
$$
a\circ x\circ b=(ax\otimes 1\otimes 1)(1\otimes \Delta)(\rho(b)),\ x\in A\otimes K\otimes K, a,b\in A.
$$
So it remains to note that $K\otimes K$ is isomorphic to $K\otimes K_{\rm vect}$
as a right $K$-module, where $K_{\rm vect}$ is $K$ with the trivial action of $K$. 

The last statement follows directly from the definitions. 
\end{proof}

\subsection{Basic facts from Galois theory}

Recall the basic setup of Galois theory (\cite{L}). 
Let $L\supset F$ be a normal algebraic field extension,
i.e. $L$ is the splitting field over $F$ 
of a family of polynomials. Let $F_{\rm perf}$ be the perfect
closure of $F$ in $L$, i.e. the field of elements whose 
$p^n$-th power is in $F$ for some $n$ if ${\rm char}F=p$ (in characteristic zero, by definition 
$F_{\rm perf}=F$). Let $L_{\rm sep}$ be the 
field of separable elements in $L$ over $F$, i.e. those whose
minimal polynomials are separable (have a nonzero derivative).  
Then $L$ is a separable extension of $F_{\rm perf}$
(see \cite{L}, Chapter 5, Proposition 6.11), 
and a purely inseparable extension of $L_{\rm sep}$, 
and $[L:F_{\rm perf}]=[L_{\rm sep}:F]$. 
Moreover, this number (finite or infinite) 
equals the order of the Galois group 
${\rm Gal}(L/F)$, which is the group of 
automorphisms of $L$ which acts trivially on $F$. 
One has $L^G=F_{\rm perf}$, $GL_{\rm sep}=L_{\rm sep}$, and 
$L_{\rm sep}^G=F$. 

\subsection{Galois and weakly Galois bimodules over fields}

\begin{proposition}\label{dimeq}
Let $P$ be a bimodule over a field $L$ which 
is finite dimensional on both sides and such that $P\otimes_L P$
is contained in a multiple of $P$. Then the right and left 
dimensions of $P$ coincide, so that $P$ is a weakly Galois
$L$-bimodule. 
\end{proposition}

\begin{proof}
Let $\dim_L P=d$, $\dim P_L=d'$. 
The bimodule $P$ has finite length (which is 
$\le {\rm min}(d,d')$). Let $\lbrace{R_m\rbrace}$ be the simple
composition factors of $P$, and let $r_-,r_+$ 
be the minimal and maximal ratio of the right dimension and the left 
dimension among the $R_m$. Since $P\otimes_L P$ is contained in
$P^N$, the bimodule $P^{\otimes s}$ is contained
in $P^{N^{s-1}}$, so the ratio 
$r_s$ of the right dimension to the left dimension of $P^{\otimes s}$ 
satisfies the inequalities $r_-\le r_s\le r_+$. 
But $r_s=(d'/d)^s$. Thus, $r_-\le (d'/d)^s\le r_+$. This implies that $d'=d$. 
\end{proof}

Let $P$ be a weakly Galois bimodule of rank $d$ over a field $L$.
Let $\phi=\phi_P: L\to {\rm Mat}_d(L)$ be the corresponding right action map. 
For each $a\in L$, let $\mu_a(t)=\sum_{i=0}^{d_a}c_i(a)t^i$ 
be the minimal polynomial 
of the matrix $\phi(a)$ over $L$ ($c_{d_a}=1$). 

\begin{proposition}\label{minpol} The coefficients 
of $\mu_a(t)$ belong to ${\mathcal Z}(P)$. 
\end{proposition} 

\begin{proof} 
Let us view $\phi$ as a map $L\to {\rm Mat}_d(k)\otimes_k L$.
Applying $\phi$ in the second component to the identity 
$$
\mu_a(\phi(a))=\sum_{i=0}^{d_a}(1\otimes c_i(a))\phi(a)^i=0, 
$$
we have 
$$
\sum_{i=0}^{d_a} (1\otimes \phi(c_i(a)))(1\otimes \phi)(\phi(a))^i=0.
$$
On the other hand, $\phi_{P\otimes_L P}(a)=(1\otimes \phi)(\phi(a))$ so 
since $P\otimes_L P$ is contained in $P^N$, we see that 
the matrix $(1\otimes \phi)(\phi(a))$ is conjugate to 
the restriction of ${\rm Id}_N\otimes \phi(a)$ to an invariant subspace. 
This implies that 
$$
\sum_{i=0}^{d_a} (1\otimes 1\otimes c_i(a))(1\otimes \phi)(\phi(a))^i=0.
$$
So, subtracting, we get
\begin{equation}\label{subt}
\sum_{i=0}^{d_a-1} 
(1\otimes (\phi(c_i(a))-1\otimes c_i(a)))(1\otimes \phi)(\phi(a))^i=0
\end{equation}
(note that the terms corresponding to $i=d_a$ 
cancel since $c_{d_a}(a)=1$). 

Let $C=L\phi(L)$ be the subalgebra of ${\rm Mat}_d(L)$ generated by $L$ 
and $\phi(L)$. Then $C$ is a commutative algebra over the field $\phi(L)$, 
and equation (\ref{subt}) can be viewed as an identity in 
${\rm Mat}_d(k)\otimes_k C\subset {\rm Mat}_d(k)\otimes_k {\rm Mat}_d(k)\otimes_k L$. 

Since $\mu_a(t)$ has degree $d_a$, the matrices 
$\phi(a)^i\in {\rm Mat}_d(k)\otimes_k L$ for $i=0,...,d_a-1$ are linearly independent over $L$. 
This implies that the matrices $(1\otimes \phi)(\phi(a)^i)\in {\rm Mat}_d(k)\otimes_k \phi(L)$
for $i=0,...,d_a-1$ are linearly independent over $\phi(L)$, hence over $C$.  
However, equation (\ref{subt}) provides a linear relation between them over $C$. 
So this relation must be trivial, which implies that 
$\phi(c_i(a))=1\otimes c_i(a)$ for all $i$, as desired.   
\end{proof} 

\begin{proposition}\label{algeb1} 
If $P$ is a weakly Galois $L$-bimodule, then 
$L$ is an algebraic extension of the center ${\mathcal Z}(P)$. 
\end{proposition}

\begin{proof}
Proposition \ref{minpol} implies that
for every $a\in L$, $\phi(a)$ satisfies a monic polynomial equation 
over ${\mathcal Z}(P)$, namely the equation $\mu_a(t)=0$.
Hence, $a$ itself satisfies the same equation.  
Thus, $L$ is an algebraic extension of ${\mathcal Z}(P)$. 
\end{proof}

\begin{proposition}
\label{semisi}
If $L$ has characteristic zero or $p>d$, 
then the matrix $\phi_P(a)$
is semisimple for all $a\in L$, i.e.,  
its minimal polynomial $\mu_a$ has distinct roots 
 (over an extension of $L$). 
\end{proposition}

\begin{proof}
Let ${\mathcal Z}={\mathcal Z}(P)$, and consider 
the subfield $F$ of $\phi(L)$ generated over ${\mathcal Z}=\phi({\mathcal Z})$ by
$\phi(a)$. Then by Proposition \ref{minpol}, $F={\mathcal
  Z}[t]/(\mu_a(t))$, so it is a finite field extension of
${\mathcal Z}$ of degree $\le d$. 
Since we are in characteristic zero or $p>d$, this extension is separable, 
so $\mu_a$ has distinct roots and $\phi(a)$ is semisimple. 
\end{proof}

\begin{remark}\label{nonsem}
Proposition \ref{semisi} is false in positive characteristic
if $p$ divides $d$. 
For example, let ${\mathcal Z}=k(t)$, where $k$ has characteristic $p$, 
let $L={\mathcal Z}[u]/(u^p-t)$, and let $P=L\otimes_{\mathcal Z}L$. Then $\phi_P(u)$ 
is the operator of multiplication by $x$ in $L[x]/(x^p-t)$, which 
is not semisimple, since $(x-u)^p=0$ but $x-u\ne 0$. 
The extension $L/{\mathcal Z}$ is not separable in this case. 
\end{remark}

\subsection{Split Galois and weakly Galois bimodules over fields}

\begin{proposition}\label{normal} 
(i) If $P$ is a split weakly Galois
$L$-bimodule then $L$ is a normal extension of ${\mathcal Z}(P)$. 

(ii) Let $L$ be a finite field extension of $F$. 
The Galois $L$-bimodule $L\otimes_F L$ is 
split if and only if $L$ is a normal extension of $F$. 

(iii) In characteristic $0$ or $p>d$ all 
split weakly Galois $L$-bimodules of rank $d$ are separable. 
\end{proposition}

\begin{proof} 
(i) By Proposition \ref{minpol}, the minimal 
polynomial of $\phi(a)\in \phi(L)$ 
over ${\mathcal Z}(P)$ is $\mu_a$. Hence,  
the minimal polynomial of $a\in L$ 
over ${\mathcal Z}(P)$ is also $\mu_a$ 
(as $\phi: L\to \phi(L)$ is an isomorphism). 
Thus, $L$ is the minimal common splitting field for the 
polynomials $\mu_a,a\in L$, so it is a normal extension. 

(ii) It is clear that the center of the $L$-bimodule $L\otimes_F L$ is $F$, 
so the ``only if'' part follows from (i). To prove the ``if''
part, note that the eigenvalues of $\phi(a)$ for $a\in L$ are 
the roots of the minimal polynomial of $a$ over $F$.
Since $L$ is normal, all the roots of this polynomial
are in $L$, i.e. $L\otimes_F L$ is a split Galois bimodule. 

(iii) This follows from Proposition \ref{semisi}.  
\end{proof}

Now we will describe the structure of split Galois and weakly Galois bimodules.
Recall Example \ref{exaa}(4).
Given a finite subgroup $G\subset {\rm Aut}(L)$ and integers $n_g>0$, 
define the split separable 
weakly Galois $L$-bimodule 
$$
P({\bold n}, L,G):=\oplus_{g\in G}(Lg)^{n_g},
$$
of rank $d=\sum_{g\in G}n_g$ (where ${\bold n}: G\to \Bbb Z$ is given by ${\bold n}(g)=n_g$). It is clear that 
${\mathcal Z}(P({\bold n},L,G))=L^G$. If $n_g=1$, then $P({\bold n},L,G)$ is a Galois $L$-bimodule,  
and we will denote it by $P(L,G)$. It is easy to see that $P(L,G)\cong L\otimes_{L^G}L$. 
 
\begin{proposition}\label{p1}
(i) All split separable weakly Galois $L$-bimodules $P$
of rank $d$ are of the form $P=P({\bold n},L,G)$ with $\sum_{g\in G} n_g=d$. 
This $L$-bimodule is Galois if and only if $n_g=r$ is a constant function,
with $r|G|=d$. Moreover, the group $G$ is completely determined by $P$.  

(ii) Let $P$ be a split Galois (respectively, weakly Galois) $L$-bimodule, 
and let ${\rm gr}(P)$ be its associated graded
bimodule under the socle filtration (as an $L$-bimodule). 
Then ${\rm gr}(P)$ is a split separable Galois (respectively, weakly Galois) 
bimodule, and it has the form given in (i).  
\end{proposition}

\begin{proof}
(i) By Lemma \ref{fil}, 
$P\cong \oplus_{i=1}^dLg_i$ for some $g_i\in {\rm Aut}(L)$. 
Let $G$ be the set of all $g_i$. 
Since $P\otimes_L P$ is contained in $P^N$, 
for any $i,j$ there is an $m$ such that 
$g_i\circ g_j=g_m$. This means that the set $G$ of all the $g_i$ 
inside ${\rm Aut}(L)$ is closed under composition, so it is a finite 
subgroup\footnote{It is clear that a finite subset $G$ 
of any group which is closed under multiplication is a subgroup. Indeed, 
if $g\in G$ then by the pigeonhole principle $g^m=g^n$ for some $m>n$, 
so $g^{m-n}=1$ and $g^{m-n-1}=g^{-1}$.}. For each $g\in G$, let $n_g>0$ 
be the multiplicity of $g$ among 
the $g_i$. Then $P=P({\bold n},L,G)$, where ${\bold n}=(n_g)$, which proves the first statement. 

Now assume that $P$ is Galois of rank $d$. 
Then the element ${\bold p}:=\sum_{g\in G}n_gg$ 
satisfies the relation ${\bold p}^2=d{\bold p}$. Also, since $\sum_g n_g=d$, 
we get ${\bold p}(\sum_g g)=d(\sum_g g)$. 
Thus by the Frobenius-Perron theorem, there exists 
a positive integer $r$ such that $\sum n_g g=r\sum_g g$, so
$n_g=r$, $P=P(L,G)^r$.

The group $G$ is the group of all $g\in {\rm Aut}(L)$ such that $Lg$ occurs in $P$, so it is completely 
determined by $P$. Part (i) is proved. 

(ii) By Lemma \ref{fil}, ${\rm gr}(P)$ 
is a split separable Galois (respectively, weakly Galois) 
$L$-bimodule, and (i) applies. 
\end{proof}

\begin{remark}
Over a field of characteristic $p>0$, 
a split Galois bimodule may fail to be separable (=semisimple). 
An example of such a Galois $L$-bimodule $P$ is given in 
Remark \ref{nonsem}. This bimodule has length $p$, and all
its composition factors are copies of the trivial bimodule $L$. 
\end{remark} 

\subsection{Behavior of split 
weakly Galois bimodules under field extensions} 
Note that if a Galois bimodule $P$ over $L$ is split, 
and if $E$ is a finite field extension 
of $L$, then the bimodule $Q:=E\otimes_L P\otimes_L E$, 
which is a Galois $E$-bimodule by Lemma \ref{ext}(i),  
does not have to be split, even if $E$ is 
Galois over $L$. 

\begin{example} Let $L={\Bbb C}(t)$ and $G=\Bbb Z/2$ 
acting on $L$ by $t\mapsto -t$. 
Let 
$$
E:=L[u]/(u^2-1-t).
$$ 
Take the Galois bimodule $P:=P(L,G)$ over $L$, of rank $2$. 
Then $Q$ has rank $4$. Consider the eigenvalues 
of the 4 by 4 matrix $\phi_Q(u)$. They are 
$\pm \sqrt{1+t}$ and $\pm \sqrt{1-t}$, and the last two don't
belong to $E$. 
Thus, $Q$ is not split over $E$. 
\end{example}

This happens because $E$ is not Galois over $L^G$ 
(it has degree 4 over $L^G$, while the Galois closure has degree 8)
and the automorphism $g$ of $L$ given by $g(t)=-t$ does 
not lift to $E$. More precisely, we have the following 
result. 

\begin{proposition}\label{p1a}
Let $G\subset {\rm Aut}(L)$ be a finite subgroup, and $E$ a
normal extension of $L$. 
Then: 

(i) The weakly Galois $E$-bimodule $Q=E\otimes_L P({\bold n},L,G)\otimes_L E$ 
is split if and only if $E$ is a normal extension of the center 
${\mathcal Z}(P({\bold n},L,G))=L^G$. 

(ii) In this case, ${\rm gr}(Q)=P(r{\bold n},E,G')$, 
where $G'={\rm Gal}(E/L^G)$ and 
$r=[E:E_{\rm sep}]$, so that we have an exact sequence
$$
1\to {\rm Gal}(E/L)\to G'\to G\to 1,
$$  
and where we denote the pullback of ${\bold n}$ to $G'$ also by ${\bold n}$. 
Moreover, if $E$ is separable over $L$, we have $Q=P({\bold n},E,G')$.  
\end{proposition} 

\begin{proof}
(i) The bimodule $P({\bold n},L,G)$ contains $P(L,G)$ and is
contained in its multiple. Also, we have 
$$
E\otimes_L P(L,G)\otimes_L E=E\otimes_L L\otimes_{L^G} L\otimes_L E=
E\otimes_{L^G}E,
$$
so (i) follows from Proposition \ref{normal}(ii). 

To prove (ii), note that $Q$ has a filtration with successive
quotients $Eg\otimes_L E$, $g\in G$ ($n_g$ copies of each), 
and $Eg\otimes_L E$ has a filtration 
whose successive quotients are $Eh$, where $h\in g\cdot {\rm
Gal}(E/L)$ (each occurring $r$ times). This implies the statement on
the structure of ${\rm gr}(Q)$. Finally, 
if $E$ is separable over $L$, then $Q$ is separable too, 
and thus $Q=P({\bold n},E,G')$.
\end{proof}

\subsection{General weakly Galois bimodules}

Let $L$ be a field and $P$ be a weakly Galois $L$-bimodule of rank $d$. 

\begin{theorem}\label{spli2}
(i) Let $E$ be the smallest field extension of $L$
that contains all the eigenvalues of $\phi(a):=\phi_P(a)$, $a\in L$.
Then $E$ is a finite extension of $L$ 
and a normal extension of ${\mathcal Z}(P)$. 
 
(ii) $E\otimes_L P\otimes_L E$ is a split weakly Galois
$E$-bimodule of rank $d[E:L]$, which is Galois if so is $P$. 

(iii) We have 
${\rm gr}(E\otimes_L P\otimes_L E)=P({\bold n},E,G)$,
where $G$ is a finite subgroup of ${\rm Aut}(E)$
containing $H:={\rm Gal}(E/L)$ 
as a (not necessarily normal) subgroup. 

(iv) $P$ contains a copy of the trivial
bimodule $L$. 

(v) If $F$ is any finite extension of $L$ 
then ${\mathcal Z}(F\otimes_L P\otimes_L F)={\mathcal Z}(P)$.
In particular, ${\mathcal Z}(F\otimes_L P({\bold n},L,G)\otimes_L F)=L^G$. 

(vi) 
The center ${\mathcal Z}({\rm gr}(E\otimes_L P\otimes_L E))$ is $L^G:=L\cap
E^G=E^G$.

(vii) $P$ is split if and only if ${\rm Gal}(E/L)$ is normal in
$G$. 
\end{theorem}

\begin{proof}
(i) Consider the commutative finite dimensional 
$L$-algebra $C=L\phi(L)\subset {\rm Mat}_d(L)$. Let $a_1,...,a_s\in L$ 
be such that $\phi(a_1),...,\phi(a_s)$ are generators of $C$ over $L$. 
Consider the field $E_*$ obtained from $L$ by adding the
eigenvalues of $\phi(a_i)$, $i=1,...,s$. Then $E_*$ is finite
over $L$, and there exists a
basis ${\mathcal B}$ of $E_*^d$ over $E_*$ in which the matrices $\phi(a_i)$ 
are upper triangular. In this basis, the matrix $\phi(a)$ 
is upper triangular for all $a\in L$, since it is a polynomial of
$\phi(a_i)$ over $L$. Thus, $E_*$ contains 
all the eigenvalues of $\phi(a)$, so $E_*=E$, and hence 
$E$ is finite over $L$. 

It is clear that $E$ is a normal extension of ${\mathcal Z}(P)$, since
it is the splitting field of the minimal polynomials of $\phi(a)$
for $a\in L$, which according to Proposition \ref{minpol} 
have coefficients in ${\mathcal Z}(P)$. 

(ii) By Lemma \ref{ext}(i), $E\otimes_L P\otimes_L E$
is a weakly Galois $E$-bimodule. 
 Let $\lambda_j: L\to E$, $j=1,...,d$,
be the eigenvalue homomorphisms for $\phi$, 
i.e. $\lambda_j(a)$ are the diagonal entries of $\phi(a)$ in 
the basis ${\mathcal B}$. Since 
$E$ is normal over ${\mathcal Z}(P)$, for any $i$, there exists 
an automorphism $g_i\in {\rm Gal}(E/{\mathcal Z}(P))$ such that 
$\lambda_i(a)=g_ia$ for any $a\in L$. This implies that 
$E^{\lambda_i}\otimes_L E=Eg_i\otimes_L E$,
which means that the associated graded 
bimodule ${\rm gr}(E\otimes_L P\otimes_L E)$ 
of $E\otimes_L P\otimes_L E$
under the socle filtration is the direct sum of 
bimodules $Eg_ih$, where $h$ runs through ${\rm Gal}(E/L)$. 
This implies that $E\otimes_L P\otimes_L E$ is a split weakly Galois
$E$-bimodule, which is Galois if so is $P$, proving (ii). 

(iii) Given (ii), Proposition \ref{p1}
applied to ${\rm gr}(E\otimes_L P\otimes_L E)$ 
implies that the set $G$ of elements $g_ih$ 
must form a finite subgroup of ${\rm Gal}(E/{\mathcal Z}(P))$ 
containing $H:={\rm Gal}(E/L)$, and
 ${\rm gr}(E\otimes_L P\otimes_L E)=P({\bold n},E,G)$.
Thus (iii) is proved.   

(iv) By (iii), for some $i=i_0$,
$\lambda_i|_L$ has to be the identity. Let us take a 
common eigenvector $v\in P$ of $\phi(a)$ with eigenvalue $\lambda_{i_0}(a)=a$: 
$\phi(a)v=av$ for all $a\in L$. Then $Lv\subset P$ is a
subbimodule isomorphic to $L$. This proves (iv). 

(v) Note first that 
$E\otimes_L P\otimes_L E$ contains $P$ as an $L$-bimodule, 
so $L\cap {\mathcal Z}(E\otimes_L P\otimes_L E)={\mathcal  Z}(P)$. 
Also note that by (iv), $E\otimes_L P\otimes_L E$ has a subbimodule 
isomorphic to $E\otimes_L E$. Hence 
${\mathcal Z}(E\otimes_L P\otimes_L E)\subset L$, 
which implies that 
${\mathcal Z}(E\otimes_L P\otimes_L E)={\mathcal Z}(P)$, 
as desired.  

(vi) This follows from (iii). 

(vii) If $P$ is split, 
Proposition \ref{p1a} implies that 
${\rm Gal}(E/L)$ is normal in
$G$. Conversely, if ${\rm Gal}(E/L)$ is normal in
$G$, we see that $L$ is a Galois extension of $L^G$, 
and ${\rm Gal}(L/L^G)=G/{\rm Gal}(E/L)$, so the statement
follows. 
\end{proof} 

\begin{example}
If $P$ has the same dimension over $L$ as a right and left vector
space but is not weakly Galois, then there may be no 
extension $E$ such that $E\otimes_L P\otimes_L E$ is split.
For example, consider $L=\Bbb C(t)$ and $P=L\oplus L$, with the
left and right actions defined by 
$$
(f\circ (a,b)\circ g)(t)=(f(t)a(t)g(t^2),f(t^2)b(t)g(t)),
$$
Then $P$ is 3-dimensional over $L$ on both sides, 
and the maps $\lambda_i$ are defined by $\lambda_1(t)=t^2$, 
$\lambda_2(t)=t^{1/2}$, $\lambda_3(t)=-t^{1/2}$. So
the splitting field $E$ has to contain $t^{1/2^n}$ for all $n$
and therefore cannot be finite over $L$.  
Note that ${\mathcal Z}(P)=\Bbb C$, 
so $L$ is a transcendental extension of ${\mathcal Z}(P)$. 
\end{example}

\begin{corollary}\label{chara}
Let $P$ be a Galois bimodule over a field $L$ of degree $d$. 

(i) Some power $\chi_a^{N_a}$ of the characteristic polynomial $\chi_a$ 
of $\phi_P(a)$, $a\in L$ is a power of its minimal polynomial $\mu_a$. 
In characteristic zero or $p>d$, one may take $N_a=1$, and in characteristic
$p$ in general, one may take $N_a=p^s$ for some $s=s_a$
such that $p^s$ divides $[L:{\mathcal Z}(P)]$.\footnote{We will
  show later in Corollary \ref{chara1} that in fact one can
  always take $N_a=1$, but this is a more difficult result.}  

(ii) The coefficients $C_i(a)$ of $\chi_a^{N_a}$
 belong to ${\mathcal Z}(P)$. 
\end{corollary}

\begin{proof}
To prove (i), let $P_*=E\otimes_L P\otimes_L E$, 
where $E$ is the splitting field 
from Proposition \ref{spli2}, and let $D=[E:L]$.  
Then for $a\in L$ the characteristic polynomial of $\phi_{P_*}(a)$ 
is the $D$-th power of the characteristic polynomial of 
$\phi_P(a)$. Also, the minimal polynomials of 
$\phi_P(a)$ and $\phi_{P_*}(a)$ coincide. 
So it suffices to check that 
some power of the characteristic polynomial of
$\phi_{P_*}(a)$ is a power of its minimal polynomial. 
But by Proposition \ref{spli2}, 
$P_*$ is split, and ${\rm gr}(P_*)=P(E,G)^r$. 
So the characteristic polynomial $\chi_a^*(t)$ of
$\phi_{P_*}(a)$ can be written in the form 
$$
\chi_a^*(t)=\prod_{g\in G}(t-ga)^r,
$$
while the minimal polynomial is 
$$
\mu_a(t)=\prod_{b\in O(a)}(t-b)^{N_a}
$$
where $O(a)$ is the orbit of $a$ under the action of $G$,
and $N_a=1$ for characteristic zero or $p>d$ and a power of $p$
dividing $[L:{\mathcal Z}(P)]$ in characteristic $p$ in general. 
This shows that $\chi_a^*(t)^{N_a}=
\mu_a(t)^K$, where $K=r\cdot 
|{\rm Stab}_G(a)|$, and ${\rm Stab}_G(a)$ is the stabilizer of
$a$ in $G$. Thus (i) is proved. 

(ii) follows from (i) and Proposition \ref{minpol}. 
\end{proof}

\subsection{Finiteness over the center} 

Let $P$ be a weakly Galois bimodule over a field $L$ of rank $d$. 

\begin{proposition}\label{algeb}
If ${\rm char}L=0$, or ${\rm char}L>d$
then $L$ is a finite extension of ${\mathcal Z}(P)$ of degree
$\le d$, whose degree divides $d$ if $P$ is a Galois bimodule.
\footnote{It will be shown later in Proposition \ref{conta3},
  Proposition \ref{conta1}, Theorem \ref{classif} that 
this result is in fact valid in any characteristic.} 
\end{proposition}

\begin{proof}
The result follows Proposition \ref{algeb1} and the Primitive Element
Theorem (\cite{L}), 
since by Proposition \ref{minpol} and Corollary \ref{chara}(i), 
every element of $L$ satisfies a polynomial equation over
${\mathcal Z}(P)$ of degree $\le d$, 
which has degree dividing $d$ if $P$ is Galois 
(namely, the equation $\mu_a(t)=0$). 
\end{proof}

\begin{corollary}\label{semisi1}
In characteristic zero or $p>d$, any weakly Galois bimodule 
of rank $d$ is semisimple as a
bimodule. 
\end{corollary}

\begin{proof}
Any weakly Galois $L$-bimodule is a bimodule over the algebra 
$L\otimes_{{\mathcal Z}(P)}L$, which by Proposition \ref{algeb}
is a finite dimensional semisimple algebra. 
\end{proof} 

\subsection{Galois bimodules containing $L\otimes_{{\mathcal Z}(P)}L$.} 

\begin{proposition}\label{conta1}
Let $P$ be a bimodule over a field $L$, finite dimensional as a
left vector space, such that $P^N$ contains $L\otimes_{{\mathcal Z}(P)}L$
for some $N$. Then $L\otimes_{{\mathcal Z}(P)}L$
is in fact a direct summand in $P^N$, and moreover in $P$.
\end{proposition}

\begin{proof}
It is clear that $[L:{\mathcal Z}(P)]<\infty$. 
Thus, $L\otimes_{{\mathcal Z}(P)}L$ is a Frobenius algebra, 
so it is an injective module over itself. Hence 
the inclusion of $L\otimes_{{\mathcal Z}(P)}L$ 
into $P^N$ splits, so $L\otimes_{{\mathcal Z}(P)}L$
is a direct summand in $P^N$. 
Furthermore, $L\otimes_{{\mathcal{Z}}(P)}L$ is a finite
dimensional commutative algebra, so it has multiplicity free 
decomposition into indecomposable projective modules.
Hence $L\otimes_{{\mathcal{Z}}(P)}L$ is a direct summand in $P$, as
desired. 
\end{proof}

\begin{proposition}\label{conta2}
Suppose that $P$ be a Galois $L$-bimodule containing 
$L\otimes_{{\mathcal Z}(P)}L$ as a direct summand.
\footnote{It follows from Proposition \ref{conta1} and 
Proposition \ref{conta3} below that 
in fact any Galois $L$-bimodule has this property.} 
Then $P$ is a multiple of 
$L\otimes_{{\mathcal Z}(P)}L$.  
\end{proposition}

\begin{proof}
Let $[L:{\mathcal Z}(P)]=m<\infty$.
We have $P=(L\otimes_{{\mathcal Z}(P)}L)^r\oplus M$, where 
$r\ge 1$, and $M$ does not have direct summands of the form 
$L\otimes_{{\mathcal Z}(P)}L$. We have $\dim M=d-rm$ as a left
and a right vector space. Since $P\otimes_L P\cong P^d$, 
and since 
$$
L\otimes_{{\mathcal Z}(P)}L\otimes_L M\cong 
M\otimes_L L\otimes_{{\mathcal Z}(P)}L\cong 
(L\otimes_{{\mathcal Z}(P)}L)^{d-rm},
$$
we have 
$$
P\otimes_L P\cong (L\otimes_{{\mathcal Z}(P)}L)^{r^2m+2r(d-rm)}\oplus
M\otimes_L M\cong P^d\cong 
(L\otimes_{{\mathcal Z}(P)}L)^{rd}\oplus M^d.
$$
Since $L\otimes_{{\mathcal Z}(P)}L$ is a commutative
algebra and hence has a multiplicity free decomposition into
projective modules over itself, $M^d$ does not contain
$L\otimes_{{\mathcal Z}(P)}L$ as a direct summand.
This implies that $rd\ge r^2m+2r(d-rm)$, i.e., $r^2m\ge rd$, or
$rm\ge d$ (as $r\ge 1$). 
Hence $M=0$, which implies the result.\footnote{This proof is based on the idea of ``projectivity defect'', 
\cite{EO}, Section 2.5.}
\end{proof}

\subsection{Purely inseparable weakly Galois bimodules} 

\begin{definition} A split 
weakly Galois $L$-bimodule 
$P$ of rank $d$ is said to be {\it purely inseparable} 
if ${\rm gr}(P)=L^d$. 
\end{definition} 

Let $P$ be a split weakly 
Galois $L$-bimodule of rank $d$. Then ${\rm gr}(P)$ is 
separable, so by Proposition \ref{p1}, 
${\rm gr}(P)=P({\bold n},L,G)=\oplus_{g\in G}(Lg)^{n_g}$, where
$G\subset {\rm Aut}(L)$ is a finite subgroup and $n_g>0$. 
Let $F=L^G$.

\begin{lemma}\label{restri}
(i) The $L$-bimodule $L\otimes_F P\otimes_F L$ 
embeds into $P^N$ for some $N$. 

(ii) The restriction of $P$ to $F$ is a purely inseparable 
weakly Galois $F$-bimodule 
of rank $d|G|$. 
\end{lemma}

\begin{proof}
(i) One has $L\otimes_F L=\oplus_{g\in G}Lg$. 
Now, for each $g\in G$ pick an eigenvector $v_g\in P$ of
the operators $\phi(a)$ with eigenvalues $g(a)$. 
Then $\oplus_{g\in G}Lv_g=\oplus_{g\in G}Lg=L\otimes_F L$. 
Thus, $L\otimes_F L\subset P$, 
so $L\otimes_F P=L\otimes_F L\otimes_L
P\subset P\otimes_L P\subset P^n$, and similarly $L\otimes_F
P\otimes_F L\subset P^{n^2}$.

(ii) We have $P\otimes_F P=P\otimes_L L\otimes_F P\subset P\otimes_L P^n
\subset P^{n^2}$. This implies that $P|_F$ 
is a weakly Galois
$F$-bimodule of rank $d|G|$. 
Moreover, $P|_F$ is split, 
and ${\rm gr}(P|_F)=L^d=F^{d|G|}$, as desired.
\end{proof}

\subsection{The restricted Lie algebra of derivations 
attached to a purely inseparable weakly Galois $L$-bimodule}

This and the next subsection are devoted to studying the
structure of purely inseparable weakly Galois bimodules. 
This is nontrivial only in characteristic $p>0$, as in
characteristic zero by Proposition \ref{semisi}, 
all purely inseparable weakly Galois
$L$-bimodules are multiples of $L$. 

Let ${\rm char}L=p>0$. 
Let $M$ be an extension of the trivial $L$-bimodule $L$ by
itself.  Such extensions are classified by ${\rm Ext}^1_{\rm
L-bimod}(L,L)={\rm Der}(L)$, the space of derivations of $L$.
For a derivation $D$ of $L$, denote by $M(D)$ the corresponding
bimodule. Namely, $M(D)=L[x]/x^2$, and the left action of $L$ is
as usual, while the right action is given by $v\circ
a=v(a+xD(a))$ (such bimodules appear in \cite{J} and are called
``self-representations of fields'').  It is easy to see that
the bimodules $M(D_1)$ and $M(D_2)$ are isomorphic if and only
if there exists $a\in L^\times$ such that $aD_1=D_2$.

Now let $P$ be a purely inseparable weakly Galois $L$-bimodule. 
Let us say that a derivation $D: L\to L$ 
is $P$-compatible if $M(D)$ is a subbimodule of $P$
(or, equivalently, of some multiple of $P$, as $M(D)$ has length $2$). 
Let ${\mathcal D}(P)$ be the set of all $P$-compatible derivations. Clearly, it is a vector space 
over $L$ under left multiplication. 

\begin{proposition}\label{jac}
(i) ${\mathcal D}(P)$ is finite dimensional over $L$. 
    
(ii) ${\mathcal D}(P)$ is closed under commutator.

(iii) ${\mathcal D}(P)$ is closed under taking $p$-th powers. 
Thus, ${\mathcal D}(P)$ is a finite-dimensional restricted 
$L$-Lie ring of derivations of $L$ in the sense of Jacobson \cite{J}. 
\end{proposition}

\begin{proof}
  (i) Consider the socle filtration of $P$: ${\bold F}_1P\subset
  {\bold F}_2P\subset...$, where ${\bold F}_1P$ is the maximal
  semisimple subbimodule, ${\bold F}_2P/{\bold F}_1P$ is the
  maximal semisimple subbimodule of $P/{\bold F}_1P$, etc. Then
  any inclusion of $M(D)$ into $P$ is actually an inclusion into
  ${\bold F}_2P$. Now, the bimodule ${\bold F}_2P$ is an
  extension of $V_2={\bold F}_2P/{\bold F}_1P\cong L^{m_2}$ by
  $V_1={\bold F}_1P\cong L^{m_1}$. Thus, it defines a linear map
  $\xi: {\rm Hom}(V_1,V_2)\to {\rm Ext}^1_{\rm L-bimod}(L,L)={\rm
    Der}(L)$.  It is clear that $M(D)$ is contained in $P$ if and
  only if $D\in {\rm Im}(\xi)$. Thus, ${\mathcal D}(P)={\rm
    Im}(\xi)$, so ${\mathcal D}(P)$ is finite dimensional.

  (ii) Let $X,Y\in {\rm Der}(L)$. Consider the tensor product
  $M(X)\otimes_L M(Y)\otimes_L M(X)\otimes_L M(Y)$.  This product
  is isomorphic to
  $L[x_1,y_1,x_2,y_2]/(x_1^2=y_1^2=x_2^2=y_2^2=0)$ with the usual
  left action of $L$, and right action given by
$$
v\circ a=v\cdot ({\rm Id}+x_1X)({\rm Id}+y_1Y)({\rm Id}+x_2X)({\rm Id}+y_2Y)(a)=
$$
$$
v(a+(x_1+x_2)Xa+(y_1+y_2)Ya+(x_1+x_2)(y_1+y_2)XYa+x_2y_1[Y,X]a+...).
$$
Now consider the 2-dimensional subspace $N$ in this product spanned by $v_1=x_1x_2y_1y_2$ and $v_2=(x_1-x_2)(y_1-y_2)$. 
Then from the last formula we get $v_1\circ a=v_1a$, and 
$v_2\circ a=v_2a+v_1[X,Y]a$. Thus, $N=M([X,Y])$. 

Now assume that $X,Y\in {\mathcal D}(P)$. Then 
$M(X)\otimes_L M(Y)\otimes_L M(X)\otimes_L M(Y)$ and hence $N$ is contained 
in $P^{\otimes 4}$, which is contained in a multiple of
$P$, since $P$ is a weakly 
Galois bimodule. Hence $N=M([X,Y])$ is contained in $P$, so $[X,Y]\in {\mathcal D}(P)$, as desired. 

(iii) Consider the tensor power $M(D)^{\otimes {2p-1}}$. 
This is the algebra 
$$
L[x_1,...,x_{2p-1}]/(x_i^2=0, i=1,...,2p-1) 
$$
with the usual left action of $L$ and right action given by 
$$
v\circ a=v\prod_{i=1}^{2p-1} ({\rm Id}+x_iD)(a)=v\sum_{j=0}^{2p-1}
e_j(x_1,...,x_{2p-1})D^j(a), 
$$
where $e_j$ are the elementary symmetric functions. Consider the vectors 
$v_1=e_{2p-1}$ and $v_2=e_{p-1}$. 
Since $\sum_{i=0}^{2p-1} e_it^i=\prod_{k=1}^{2p-1} 
(1+tx_k)$, we have 
$$
e_ie_j=\binom{i+j}{j}e_{i+j}, 
$$
so $e_{p-1}e_j=\binom{p-1+j}{j}e_{p-1+j}=\frac{p(p+1)...(p+j-1)}{j!}e_{p-1+j}$. 
This means that $e_{p-1}e_j=0$ for $j=0,...,p-1$, but $e_{p-1}e_p=e_{2p-1}$. 
So we get $v_1\circ a=v_1a$ and $v_2\circ a=v_2a+v_1D^pa$. 
Thus, $M(D^p)$ is contained in $M(D)^{\otimes 2p-1}$.
Thus, if $M(D)$ is contained in $P$, then $M(D^p)$ is contained in 
$P^{\otimes 2p-1}$, which is a multiple of $P$. 
Thus, $M(D^p)$ is contained in $P$ and hence $D^p\in {\mathcal D}(P)$, 
as desired.   
\end{proof} 

\subsection{The containment theorem for 
purely inseparable weakly Galois $L$-bimodules}

\begin{theorem}\label{insep}
Let $P$ be a purely inseparable weakly Galois $L$-bimodule of rank $d$. 
Then $L$ is a finite purely inseparable extension of ${\mathcal Z}(P)$.  
Moreover, $P$ contains the bimodule $L\otimes_{{\mathcal Z}(P)}L$.
\end{theorem}

\begin{proof}
The theorem is trivial in characteristic zero, so we will assume
that ${\rm char}L=p>0$. 

We will prove the theorem by induction in the length $\ell(P)$ of the socle
filtration of $P$, starting from the trivial case $\ell=1$. 

By Jacobson's theorem \cite{J}, if ${\mathcal D}$ 
is a finite dimensional restricted 
$L$-Lie ring of derivations of $L$, then $[L:L^{\mathcal D}]
<\infty$. Therefore, Proposition \ref{jac}   
implies that $L$ is a finite extension 
of the field of invariants $F:=L^{{\mathcal D}(P)}$, of exponent $1$ (i.e., 
$L^p\subset F$). Let $[L:F]=p^m$ and $L=F(x_1,...,x_m)$, where
$x_i^p=\alpha_i\in F$.
Define derivations $D_1,...,D_m$ of $L$ over $F$ 
by $D_i(x_j)=\delta_{ij}$.  

Let $D$ be a derivation of $L$ such that $D^p=0$. 
Consider the $L$-bimodule $M(D)^{\otimes {p-1}}$, and consider the subbimodule $Q(D)$ in it 
generated by $1$. As explained in the proof of Proposition
\ref{jac}, 
this subbimodule is spanned over $L$ by $e_i$, $i=0,...,p-1$ (where $e_0=1$), 
and the right action is defined by 
$$
v\circ a=v\sum_{i=0}^{p-1}e_iD^i(a). 
$$
Note that $e_i=\frac{u^i}{i!}$, where $u=e_1$, so $Q(D)=L[u]/(u^p)$, and 
we can rewrite the last formula as 
$$
v\circ a=v\exp(uD(a)).
$$ 
Thus, $Q(D_1)\otimes_L...\otimes_L Q(D_m)=L[u_1,...,u_m]/(u_1^p=...=u_m^p=0)$, 
with the usual left action of $L$ and the right action given by 
$v\circ a=v\gamma(a)$, where $\gamma(x_i)=x_i+u_i$. 
Thus, $Q(D_1)\otimes_L...\otimes_L Q(D_m)\cong L\otimes_F L$ as 
$L$-bimodules, Now, $Q(D_i)$ is contained in a tensor power of
$P$. Hence,  $Q(D_1)\otimes_L...\otimes_L
Q(D_m)\otimes_L P$ is contained in a tensor power of $P$.  
Therefore, since $P$ is a weakly Galois bimodule, 
$Q(D_1)\otimes_L...\otimes_L Q(D_m)\otimes_L P$
is contained in a multiple of $P$, i.e. 
$L\otimes_F P$ is contained in $P^N$ for some $N$. 
Hence, $P\otimes_F P=P\otimes_L L\otimes_F P\subset P\otimes_L
P^N\subset P^{nN}$. This implies that $P|_F$ is a 
purely inseparable 
weakly Galois $F$-bimodule with the same center ${\mathcal
  Z}(P)$. 

Moreover, we have $\ell(P|_F)<\ell(P)$. Indeed, 
if ${\bold F}_\bullet$ is the socle filtration of $P$ then 
by construction ${\bold F}_2P|_F$ is semisimple (i.e., 
isomorphic to a multiple of $F$). 

Thus, by the induction assumption, 
$F$ is a finite purely inseparable extension 
of ${\mathcal Z}(P)$. 
It follows that $L$ is a finite purely inseparable
extension of ${\mathcal Z}(P)$. 

So to complete the induction and prove the theorem, it remains 
to show that $P$ contains $L\otimes_{{\mathcal Z}(P)}L$. 

We have seen that $L\otimes_F L$ embeds into a tensor power, 
hence a multiple of $P$. 
Since $P$ is weakly Galois, this means that 
$L\otimes_F P\otimes_F L=L\otimes_F L\otimes_L P\otimes_L
L\otimes_F L$ embeds into a multiple of $P$.
But by the induction assumption, $P$ contains 
$F\otimes_{{\mathcal Z}(P)}F$, hence  
$L\otimes_F P\otimes_F L$ contains $L\otimes_{{\mathcal Z}(P)}L$. 
Hence $L\otimes_{{\mathcal Z}(P)}L$ is contained in $P^N$ for some $N$. 
Thus, by Proposition \ref{conta1}, 
$L\otimes_{{\mathcal Z}(P)}L$ is actually a direct 
summand in $P$.
The theorem is proved. 
\end{proof} 

\subsection{The classification theorem for Galois bimodules over
 a field}

\begin{proposition}\label{conta3} 
Let $P$ be a weakly Galois $L$-bimodule.
Then the bimodule $L\otimes_{{\mathcal{Z}}(P)}L$ 
is contained in $P^N$ for some $N$. 
In particular, $[L:{\mathcal Z}(P)]<\infty$.  
\end{proposition}

\begin{proof}
First assume that $P$ is split. 
Let $G$ be the set of $g\in {\rm Aut}(L)$ such that 
$Lg$ occurs in ${\rm gr}(P)$, 
and let $F=L^G$. Then by Lemma \ref{restri}(ii), 
$P|_L$ is a purely inseparable weakly Galois 
$F$-bimodule, with the same center as $P$. 
So by Theorem \ref{insep} it contains
$F\otimes_{{\mathcal Z}(P)}F$. 
This means that $L\otimes_F P\otimes_F L$ 
contains $L\otimes_{{\mathcal{Z}}(P)}L$. 
But we know from Lemma \ref{restri}(i)  
that $L\otimes_F P\otimes_F L$ embeds into a multiple of
$P$. Thus, $L\otimes_{{\mathcal{Z}}(P)}L$ embeds into a multiple of $P$.   

Now consider the general case. Let $E$ be the extension of 
Theorem \ref{spli2}, such that the $E$-bimodule $E\otimes_L P\otimes_L E$ is 
split. Then, as we have just shown, a multiple of 
$E\otimes_L P\otimes_L E$ contains 
$E\otimes_{{\mathcal Z}(P)}E$, which in turn contains 
$L\otimes_{{\mathcal Z}(P)}L$. 
But as an $L$-bimodule, $E\otimes_L P\otimes_L E$ is
a multiple of $P$. So we get that 
$L\otimes_{{\mathcal Z}(P)}L$ is contained in a multiple 
of $P$, as desired.  
\end{proof} 

Now we are ready to prove the main theorem about Galois
bimodules. 

\begin{theorem}\label{classif} 
Let $L$ be a field. Then any Galois $L$-bimodule $P$ of rank $d$ 
is a multiple of $L\otimes_{{\mathcal Z}(P)}L$. 
In particular, $[L:{\mathcal Z}(P)]$ is finite and divides $d$. 
\end{theorem}
 
\begin{proof}
By Proposition \ref{conta3}, $L\otimes_{{\mathcal{Z}}(P)}L$ 
is contained in $P^N$ for some $N$.
By Proposition \ref{conta1}, this means that 
$L\otimes_{{\mathcal{Z}}(P)}L$ is a direct summand in $P$. 
By Proposition \ref{conta2}, this means that 
$P$ is a multiple of $L\otimes_{{\mathcal{Z}}(P)}L$, which
implies the theorem. 
\end{proof}

\begin{corollary}\label{chara1}
In Corollary \ref{chara}, one may take $N_a=1$. 
Thus, if $P$ is a Galois $L$-bimodule, then 
the characteristic polynomial of $\phi_P(a)$ 
is a power of its minimal polynomial, and 
the coefficients $C_i(a)$ of this characteristic polynomial are
in ${\mathcal Z}(P)$. 
\end{corollary}

\begin{proof}
By Theorem \ref{classif}, we only need to show 
that the statement holds for $P=L\otimes_{{\mathcal Z}(P)}L$. 
But in this case the minimal polynomial $\mu_a$ 
and the characteristic polynomial $\chi_a$  of
$\phi(a)$ are just the minimal and characteristic polynomials 
of the operator $M_a$ of multiplication by $a$ 
in $L$, regarded as a ${\mathcal Z}(P)$-vector space. 
So, it is clear that $\chi_a=\mu_a^r$, where  
$r=[L:{\mathcal Z}(P)(a)]$. 
\end{proof}

\subsection{Galois bimodules over matrix algebras and central
  simple algebras}

Let $L$ be a field. 

\begin{proposition}\label{matr} There is a natural bijection
between isomorphism classes of 
Galois bimodules over $L$ of rank $d$ and 
isomorphism classes of Galois bimodules over ${\rm Mat}_m(L)$ 
of rank $d$, defined by $P\mapsto {\rm Mat}_m(P)$
as in Lemma \ref{ext}(ii).
The center is preserved under this bijection.   
\end{proposition}

\begin{proof}
We only need to construct the inverse. Suppose that $Q$ is a
Galois bimodule for ${\rm Mat}_m(L)$ of rank $d$.
Define the $L$-bimodules $\bar Q_{ij}=E_{ii}QE_{jj}$, where 
$E_{pq}$ are the elementary
matrices. Clearly, these bimodules are all isomorphic, so we'll
call them $\bar Q$. We clearly have 
$\bar Q\otimes_L\bar Q=\bar Q^d$, 
and also $m\bar Q=\oplus_i \bar Q_{ij}=QE_{jj}$, which is 
of dimension $dm$ over $L$ on the right.
Hence $\bar Q$ is of dimension $d$
over $L$ on the right, and similarly on the left. 
The assignment $Q\mapsto \bar Q$ is inverse to $P\mapsto {\rm
  Mat}_m(P)$, as desired.   

The fact that the center is preserved is clear. 
\end{proof}

\begin{corollary}\label{classmatr}
Any Galois bimodule over ${\rm Mat}_m(L)$
is a multiple of ${\rm Mat}_m(L\otimes_{{\mathcal{Z}}(P)}L)$.
\end{corollary}

\begin{proof}
This follows from Proposition \ref{matr} and Theorem
\ref{classif}. 
\end{proof} 

\begin{definition} 
We will say that a Galois bimodule over ${\rm Mat}_m(L)$ 
is {\it split} if so is the corresponding $L$-bimodule under the
bijection of Proposition \ref{matr}.
\end{definition} 

Now assume that $B$ is a central simple algebra over $L$
of dimension $m^2$, and let
$P$ be a Galois bimodule over $B$ of rank $d$. 
Recall that by restricting from $B$ to $L$,
by Lemma \ref{ext}(iii) $P$ is automatically a Galois $L$-bimodule
of rank $dm^2$, with the same center. 

\begin{proposition}
\label{censim}
There exists a normal extension  $K$ of 
${\mathcal Z}(P)$ finite over $L$ which is a splitting field for 
the algebra $B$ and contains the eigenvalues of $\phi_P(a)$ for all $a\in L$. 
Moreover, $K\otimes_L P\otimes_L K$ 
is a split Galois bimodule 
over ${\rm Mat}_m(K)$ of rank $d[K:L]$. 
\end{proposition}

\begin{proof} 
Let $S$ be a finite extension of $L$ which splits $B$,
and let $E$ be the extension of Theorem \ref{spli2}. 
Let $K$ be the normal closure over 
$S\cdot E$ over $E$.  
Then by Theorem \ref{spli2}, Proposition \ref{p1a} 
and Proposition \ref{matr}, the $K$-bimodule 
$K\otimes_L P\otimes_L K$ is a split Galois bimodule 
over ${\rm Mat}_m(K)$ of rank $d[K:L]$. 
 \end{proof}
  
\begin{corollary}\label{censim1}
If $P$ is a Galois bimodule over $B$, then 
$P^{m^2}$ is a multiple of $B\otimes_{{\mathcal Z}(P)}B$.  
\end{corollary} 

\begin{proof}
By Proposition \ref{censim} and Proposition \ref{classmatr}, 
the Galois ${\rm Mat}_m(K)$-bimodule $K\otimes_L P\otimes_L K$
has the form ${\rm Mat}_m(K\otimes_{{\mathcal Z}(P)}K)^r$.
So $K\otimes_L P^{m^2}\otimes_L K$ 
has the form $({\rm Mat}_m(K)\otimes_{{\mathcal
    Z}(P)}{\rm Mat}_m(K))^r$, which as a $B$-bimodule 
is $(B\otimes_{{\mathcal Z}(P)}B)^{rn^2}$, where 
$n=[K:L]$. On the other hand, 
$$K\otimes_L P^{m^2}\otimes_L K=P^{m^2n^2}$$ as a $B$-bimodule. 
So we get that $P^{m^2}=(B\otimes_{{\mathcal Z}(P)}B)^r$, as
desired. 
\end{proof}
\section{Invariants of 
$K$-comodule algebras finite over center}

\subsection{The main theorem}

Let $K$ be a finite dimensional Hopf algebra over an algebraically closed field $k$. 
Let $A$ be a (right) $K$-comodule algebra. 
We are interested in the $K$-invariants $A^K$ of $A$, i.e., 
the space of elements $a\in A$ such that $\rho(a)=a\otimes 1$. 
\footnote{We note that  many authors call these elements coinvariants and denote the space of such elements by $A^{{\rm co}K}$.}

 Suppose that $Z$ is a central $k$-subalgebra 
of $A$ (not necessarily $K$-costable\footnote{A $K$-costable subspace of a $K$-comodule is the same thing as a subcomodule.}) which is an integral domain, 
and let $Q$ be the quotient field of $Z$. 

Consider the following (redundant) list of assumptions
on $A$ and $Z$: 

(1) $A$ is finitely generated as a $Z$-module.

(2) $Q\otimes_Z A$ is a central simple algebra with center $Q$. 

(3) $A$ is a torsion-free $Z$-module. 

(4) $Z$ is integrally closed (in $Q$). 

(5) $A$ is a projective $Z$-module. 

Our main result is the following theorem. 

\begin{theorem}\label{maint}
(i) Under assumptions (1)-(4), $Z$ and hence $A$ are integral 
over $Z\cap A^K$. 

(ii) The conclusion of (i) also holds if conditions (3), (4) are
replaced with condition (5). In particular, 
it holds if $A$ is an Azumaya algebra over $Z$. 

(iii) If in addition $Z$ is a finitely generated 
algebra over $k$, 
then in the situation of (i) and (ii), 
$A$ and $Z$ are finitely generated modules over 
$Z\cap A^K$, and thus $A$ is a finitely generated module
over $A^K$ (on both sides).   

(iv) If $Z$ is a finitely generated algebra over $k$ then 
in the situation of (i) or (ii), so is $Z\cap A^K$. 
\end{theorem}

Theorem \ref{maint} is proved in the next subsection. 
We note that the proof of (i) uses only the material up to
Proposition \ref{minpol},
and the proof of (ii) only uses the material up to Corollary
\ref{chara} (inclusively). Parts (iii) and (iv) follow easily. 

\begin{remark}
It is well known that a coaction of a noncommutative 
Hopf algebra $K$ on an algebra $A$ does not have to preserve its
center. The classical example is $K=A=k[G]$ being the group algebra 
of a nonabelian finite group $G$, coacting on itself by its
coproduct. 
\end{remark}

\begin{remark} If $A$ is commutative, 
conditions (1)-(4) 
imply that $A=Z$. Indeed, if $A$ 
is torsion-free and finite over $Z$, 
and coincides with $Z$ after localization, then $A$ 
is contained in the integral closure of $Z$, 
which is $Z$ if $Z$ is integrally closed. 
Also, it is clear that (1),(2) and (5) 
imply that $A=Z$ if $A$ is commutative.

Also note that the algebra $A$ is necessarily a PI algebra, since
by (2) and either (3) or (5) it embeds into the central simple algebra
$Q\otimes_Z A$. 
\end{remark}

\begin{remark} In the case $A=Z$ Theorem \ref{maint} says that 
$Z$ is integral over $Z^K$, which is finitely generated 
as a $Z^K$-module if $Z$ is finitely generated over $k$.  
This is a result of Skryabin (\cite{S1}) in the case of 
$Z$ being a domain (Skryabin proves the result more generally, 
when $Z$ has no nilpotent $K$-costable ideals, and always in
positive characteristic). 
\end{remark}

\begin{example}
Let $A$ be a flat family of finite dimensional algebras 
parametrized by an irreducible affine algebraic variety $X$ over $k$,
such that the fiber $A_x$ is a matrix algebra for a generic 
point $x\in X$ (``flat family'' means that $A$ is projective 
over $Z:={\mathcal O}(X)$). Suppose a finite dimensional Hopf
algebra $K$ coacts on $A$, not necessarily preserving $Z$. 
Then conditions (1),(2),(5) are
satisfied, and parts (ii),(iii),(iv) of Theorem \ref{maint}
apply. Also, these statements apply to the case when 
$A={\rm Mat}_m(Z)$, where $Z$ is a finitely generated integral
domain over $k$. This is an interesting example even if $Z$ is
a polynomial algebra. 
\end{example}

\begin{corollary}
Let $Z$ be a finitely generated integrally closed domain over
$k$, and $A$ be a prime PI algebra with center $Z$ and a coaction
of $K$. Then $Z$ and $A$ are finitely 
generated modules over $Z\cap A^K$, and
$A$ is a finite module over $A^K$ on both sides.
Moreover, in this case $Z\cap A^K$ is a finitely generated
$k$-algebra. 
\end{corollary}

\begin{proof}
Condition (1) of Theorem \ref{maint} is satisfied by 
\cite{McR}, Proposition 13.6.11. 
Condition (2) holds by Posner's theorem, 
\cite{McR}, Theorem 13.6.5. 
Condition (3) holds because $A$ is prime.
Condition (4) is one of the assumptions. 
So Theorem \ref{maint} applies.  
\end{proof}

\subsection{Proof of Theorem \ref{maint}}

Let us prove (i). Let $Q=Q_Z$ be the field of quotients of $Z$.
Consider the tensor product $Q\otimes_Z P$. 
This is a left module over the algebra
$B:=Q\otimes_Z A$, which is a central simple algebra 
over $Q$ by (2), of rank $m^2$ for some $m$. 
Consider the right action of $Z$ 
on $Q\otimes_Z P$. This action defines a homomorphism 
$\psi_P: Z\to {\rm End}_Q(Q\otimes_Z P)\cong {\rm
  Mat}_{dm^2}(Q)$. 

We claim that for any nonzero $z\in Z$, the matrix 
$\psi_P(z)$ is invertible. 
Indeed, let $x\in Q\otimes_Z P$ be such that $\psi_P(z)x=xz=0$. 
Let $x=w^{-1}y$, where $w\in Z$, $w\ne 0$, and $y\in P$. 
Hence, $yz$ is a 
torsion element of $P$, i.e. 
$z'yz=0$ in $P$ for some
nonzero $z'\in Z$. But by (3) (or by (5)), 
$A$ is torsion-free over $Z$, which implies 
that $P$ is torsion-free over $Z$ on each side.
So we get that $yz=0$ and hence $y=0$ and $x=0$, 
as desired. 

Thus, we see that the right action of $Z$ on 
$Q\otimes_Z P$ naturally extends to a right action of $Q$, i.e. 
$Q\otimes_Z P$ is a $Q$-bimodule. Hence, $Q\otimes_Z P$ is a 
bimodule over the central simple algebra $B$, and 
$$
Q\otimes_Z P=Q\otimes_Z P\otimes_Z Q=B\otimes_A P\otimes_A B.
$$  

Now, let $\dim K=d$. 
By Proposition \ref{coa}, $P:=A\otimes K$
is a Galois bimodule over $A$ of rank $d$.
This implies by Lemma \ref{ext}(i) 
that $P_{\rm loc}:=Q\otimes_Z P=B\otimes_A P\otimes_A B$ 
is a Galois bimodule over $B$ of rank $d$, and hence by
restriction (Lemma \ref{ext}(iii)) a Galois
bimodule over $Q$ of rank $dm^2$, with the same center. 

By Proposition \ref{minpol}, this implies that the coefficients of the
minimal polynomial $\mu_a$ of $\phi_P(a)$ for all $a\in Q$ are in 
${\mathcal Z}(P_{\rm loc})$. 

Now we will use the following well known lemma. 

\begin{lemma}\label{inclo}
Let $Z$ be an integrally closed domain with field of quotients $Q$, and $M$ a finitely generated $Z$-module. Let 
$b: M\to M$ be an endomorphism. Then the coefficients of the
minimal polynomial and the coefficients 
of the characteristic polynomial of $b$ on $Q\otimes_Z M$ belong to
$Z$. 
\end{lemma}

\begin{proof}
Let $m_1,...,m_n$ be generators of $M$. 
Then $b(m_i)=\sum b_{ij}m_j$ for some $b_{ij}\in Z$. 
Let $B=(b_{ij})$, and let $\chi_B$ be the characteristic polynomial of $B$. 
Then by the Hamilton-Cayley theorem, $\chi_B(b)=0$. So the eigenvalues of 
$b$ on $Q\otimes_Z M$ (which are elements of $\bar Q$) are
integral over $Z$. 
Hence the coefficients of the minimal polynomial $\mu_b$ 
and the characteristic polynomial $\chi_b$ of $b$ 
(which may have lower degree than $\chi_B$) are integral over $Z$
(as they are polynomials of the eigenvalues). 
Since these 
coefficients are in $Q$ and since $Z$ is integrally closed, they are in $Z$, as desired.  
\end{proof}

We can now apply Lemma \ref{inclo} to $M=P$ as a 
left $Z$-module,
and $b=\psi_P(z)$. Then we get that the 
coefficients of the minimal
polynomial of $\psi_P(z)$ belong to 
$Z$. 

Moreover, if $z\in Z$ is central for $P_{\rm loc}$ 
then by (3) it is also central for $P$ 
(indeed, since $P$ is torsion-free over $Z$, 
the map $P\to P_{\rm loc}$ is an embedding of $Z$-bimodules). 

Thus, the coefficients of the minimal polynomial 
of $\psi_P(z)$ belong to $Z\cap {\mathcal Z}(P)$. 
By Proposition \ref{coa}, this means that 
these coefficients belong to $Z\cap A^K$.
Hence $Z$ is integral over $Z\cap A^K$
(as $z$ is annihilated by the minimal polynomial of $\psi_P(z)$). 
Thus, (i) is proved. 

Now we prove (ii). 
We use Corollary \ref{chara}(ii)
to conclude that the coefficients of some power of the characteristic
polynomial of $\psi_P(z)$ in ${\rm End}_Q(Q\otimes_Z P)\cong {\rm
  Mat}_{dm^2}(Q)$ belong to $F:={\mathcal Z}(P_{\rm loc})\subset
Q$.\footnote{In fact, Corollary \ref{chara1} (which is more
  difficult to prove than Corollary \ref{chara})
says that the coefficients of the characteristic polynomial 
itself are in $F$, but we don't need it here.}

Since condition (5) is satisfied, the module $P$ is locally free as a left
$Z$-module, so the matrix $\psi_P(z)$ actually has coefficients 
in $Z$. This implies that the coefficients of the characteristic 
polynomial of $\psi_P(z)$ belong to $Z\cap F=Z\cap A^K$
(as by (5) any element of $Z$ central in $P_{\rm loc}$ is also central
in $P$). Thus, $Z$ is integral over $Z\cap A^K$, 
and (ii) is proved.   

Part (iii) follows from (i) and (ii) and the standard fact that 
if $R\subset S$ is an integral extension of $k$-algebras, 
and $S$ is finitely generated as a $k$-algebra, then $S$ is a finitely 
generated $R$-module. 

Finally, part (iv) follows from part (iii) and the Artin-Tate lemma: 
if $B\subset C$ are commutative $k$-algebras
and $C$ is finitely generated over $k$ and finite as a module
over $B$, then $B$ is a finitely generated $k$-algebra.  
 
Theorem \ref{maint} is proved. 

\begin{remark}
1. Note that if $Z=Q$ is a field and $A$ a central simple
algebra over $Z$, then (i) follows immediately 
from Proposition \ref{coa} and Proposition \ref{minpol},
by regarding $A\otimes K$ as a Galois $Z$-bimodule.  

2. The coefficients of the minimal polynomial of a
square matrix over $Z$ (unlike those of its characteristic
polynomial) do not have to be in $Z$ 
if $Z$ is not integrally closed. For instance, take 
the 5 by 5 matrix over $Z:=\Bbb C[x^2,x^3]$ 
given by $a_{12}=1$, $a_{21}=x^2$, $a_{34}=1$, $a_{45}=1$,
$a_{53}=x^3$, and the rest of the entries are zero. 
Then the minimal polynomial is $(t^3-x^3)(t+x)$, whose 
coefficients are not in $Z$. This is why we used characteristic
polynomials rather than minimal polynomials 
in the proof of Theorem \ref{maint}(ii). 
\end{remark}

\subsection{Divisibility}

\begin{proposition}\label{divisi} In 
the situation Theorem \ref{maint} (i) or (ii), the degree 
of $Q$ over the field of quotients $Q^K$ of $Z\cap A^K$
divides the dimension of $K$.
\end{proposition}

\begin{proof} By Proposition \ref{censim}, 
for a suitable extension $E$ of $Q$, the bimodule
$E\otimes_Q P_{\rm loc}\otimes_Q E=E\otimes_Z P\otimes_Z E$ 
is a split Galois bimodule over 
${\rm Mat}_m(E)$ of rank $d[E:Q]$, where $d={\rm dim}(K)$. 
Hence by Proposition \ref{matr}, it corresponds 
to a split Galois bimodule $P'$ over $E$ of the same rank. 
The center of $P'$ is ${\mathcal Z}(P')=
{\mathcal Z}(P_{\rm loc})=Q^K$, so by Theorem 
\ref{classif}, $P'=(E\otimes_{Q^K}E)^r$, so 
$d[E:Q]=r[E:Q^K]$. Thus, $d=r[Q:Q^K]$ and hence $[Q:Q^K]$ 
divides $d$. 
\end{proof}

\begin{remark} If ${\rm char}L=0$ or ${\rm char}L=p>d$, 
then there is a simpler proof of Proposition \ref{divisi}.
Namely, in this case by Proposition \ref{normal}(iii) and 
Proposition \ref{p1}, $P'=P(E,G)^r$ for some finite
group $G\subset {\rm Aut}(E)$. We have $E^G={\mathcal
  Z}(P')=Q^K$, so $|G|=[E:Q^K]$ and hence 
$d[E:Q]=r|G|=r[E:Q^K]$. Thus $d=r[Q:Q^K]$ and hence $[Q:Q^K]$ 
divides $d$. 
\end{remark}

\subsection{Examples}

\begin{example}
Condition (2) cannot be removed to allow semisimple (rather than
simple) algebras over $Q$. Indeed, take $Z=\Bbb C[t]$ and 
$A=\Bbb C[t]\oplus \Bbb C[t]$, containing $Z$ as the
diagonal. Consider the action of $G=\Bbb Z/2$ on $A$ by 
$g(f_1(t),f_2(t))=(f_1(-t),f_2(1-t))$ for 
the generator $g\in G$ (which corresponds to a coaction of
${\rm Fun}(G)$). Then $Z\cap A^G=\Bbb C$
(it consists of polynomials $f$ such that $f(t)=f(-t)=f(1-t)$). 

However, there is a generalization of Theorem \ref{maint} 
in which $Z$ is a reduced algebra (not necessarily a domain),
$Q$ is its total quotient ring, and $Q\otimes_Z A$ is a 
semisimple algebra with center $Q$ (see Theorem \ref{maint1}.
\end{example}

\begin{example}
Condition (4) cannot be removed. Indeed, take $A=\Bbb C[x,y]$ and
$Z=\lbrace{f\in A|f(0,y)=f(1,y)\rbrace}$. Define an action 
of $G=\Bbb Z/2$ on $A$ by $(gf)(x,y)=f(x,x-y)$ for the generator
$g$ of $G$. Then $Z\cap A^G$ consists 
of polynomials $f$ such that 
$f(0,-y)=f(0,y)=f(1,y)=f(1,1-y)=f(0,1-y)$. For such $f$, we have
that $f(0,y)$ is constant. Therefore, $A$ and $Z$ cannot be
finite modules over 
$Z\cap A^G$. 

Also, condition (3) cannot be removed. Namely, 
let $R$ be any commutative
algebra over $k$ with generators $x_1,...,x_n$,  
$A=k[x_1,...,x_n,t]\oplus R$, and $Z\subset A$ 
be the image of $k[x_1,...,x_n,t]$ 
under the map $k[x_1,...,x_n,t]\to A$ given by $x_i\mapsto (x_i,x_i)$, 
$t\mapsto (t,0)$. Suppose that $K$ coacts on $R$ 
and coacts trivially on $k[x_1,...,x_n,t]$; then it 
coacts diagonally on $A$ and conditions (1),(2),(4) 
(although not (3)) are satisfied. 
Yet, the invariants in $Z=k[x_1,...,x_n,t]$ 
are those elements whose images in $R$
are invariant under the coaction of $K$ on $R$. 
So if $A$ is integral over its invariants, then 
so must be $R$. But it is known that there exist 
finitely generated commutative $k$-algebras 
with a coaction of $K$ which are not integral
over their invariants, see \cite{Z}.  
\end{example}

\begin{example}
If $A$ is a deformation of an algebra $A_0$ 
satisfying the conditions of Theorem \ref{maint} 
with a coaction of $K$, 
$A$ does not have to be finite over $A^K$, even if $Z\subset A_0$
is finitely generated. 
Indeed, let $K=H^*$, where $H$ is 
the Nichols Hopf algebra of dimension $16$ over $\Bbb C$ generated by a
grouplike element $g$ such that $g^2=1$ and 
skew-primitive elements $x_i$, $i=0,1,2$, 
such that $gx_i=-x_ig$, $x_ix_j=-x_jx_i$, $x_i^2=0$, 
and $\Delta(x_i)=1\otimes x_i+x_i\otimes g$. Let $B$ be any
$\Bbb C$-algebra. Define a right $B$-linear action of $H$ on $B^2$ 
with right $B$-basis $e_1,e_2$ by 
$$
g(e_1)=e_1,\ g(e_2)=-e_2,\ x_i(e_1)=0,\ x_0(e_2)=e_1, 
x_1(e_2)=xe_1,\ x_2(e_2)=ye_1,
$$
where $x,y\in B$ are any elements. 
Now consider the corresponding adjoint action of $H$ 
on $A:={\rm End}(B^2)_B={\rm Mat}_2(B)$, via
$$
a\circ M=a_{(1)}MS(a_{(2)})
$$
(using Sweedler's notation). Then it is easy to show that   
$A^K$ is the set of matrices $b\cdot {\rm Id}$, where $b\in Z_{x,y}$, 
and $Z_{x,y}$ is the centralizer of $x,y$ in $B$. 

Now take $B$ to be the Weyl algebra generated by $x,y$ with the
defining relation $[y,x]=1$. Then we get that $Z_{x,y}=\Bbb C$,
so the invariants $A^K$ in $A$ are trivial, and hence $A$ 
is not finite over $A^K$. 

Note, however, that $A$ can be viewed as a deformation of 
an algebra $A_0$ satisfying the conditions of Theorem
\ref{maint}, with finitely generated $Z$. 
Namely, we can consider the filtration on $A$ given by 
${\rm deg}(bE_{ij})=i-j+{\rm deg}(b)$, where $E_{ij}$ are the
elementary matrices and ${\rm deg}(b)$ is the Bernstein filtration 
degree of $b\in B$ defined by $\deg(x)=\deg(y)=1$. This filtration is preserved by $H$, and 
$A_0={\rm gr}(A)={\rm Mat}_2(\Bbb C[x,y])$
satisfies conditions (1),(2) and (5), for $Z=\Bbb C[x,y]$. 
Also, we may consider the reductions of $A$ modulo primes $p$, 
and these reductions satisfy (1),(2), and (5) for $Z=k[x^p,y^p]$. 

Thus, this is an example of a situation when the invariants of 
$A_0$ don't lift to the deformation. 

We note that such a thing could not happen for a coaction on $B$
itself: it is shown in \cite{EW1} that any action 
of a finite dimensional Hopf algebra $H$ on $B$ preserving the Bernstein filtration 
necessarily factors through a finite group algebra and hence 
$B$ is finite as a module over $B^K$, where $K=H^*$. 
Also nothing like this can happen
for a semisimple $H=K^*$ since 
it is known that in this case 
a Noetherian algebra $A$ is a Noetherian
$A^K$-module (\cite{M}, Theorem 4.4.2). 

\begin{remark} Taking $B$ to be the algebra of matrices of any
size $N$ and $x,y$ generic elements of this algebra, we get an
example when $H=K^*$ is fixed and the dimension of
$A^K$ is $1$, but the dimension of $A$ (which is a matrix algebra
of size $2N$) is arbitrarily large. 

Such an example is only possible for nonsemisimple $H$, due to
the following lemma. 

\begin{lemma}
Let $H=K^*$ be a semisimple Hopf algebra over an algebraically closed
field $k$, and $A$ 
be a semisimple $H$-module algebra. Then 
$\dim A^K\cdot \dim H\ge \dim A$. 
\end{lemma}

\begin{proof}
Consider the algebra $H\# A$. This algebra is
semisimple (\cite{M}, Theorem 7.4.2(2)). 
Let us decompose $A$ as an $H\# A$-module:
$A=\oplus n_iV_i$, where $V_i$ are the 
simple $H\# A$-modules, and $n_i\ge 0$. 
Then $A^K={\rm End}_{H\# A}(A)=\oplus_i {\rm Mat}_{n_i}(k)$,
so $\dim A^K=\sum_i n_i^2$. 
Also, if $d_i=\dim V_i$, then we have 
$\dim H\dim A=\sum d_i^2$ and $\dim A=\sum n_id_i$. 
Thus, the result follows from the Cauchy-Schwarz inequality. 
\end{proof}
\end{remark}
\end{example}

\begin{example}\label{galgro} (\cite{EW2}) The following example 
of an inner faithful coaction of a non-basic finite dimensional Hopf algebra on a field
shows that in the situation of Theorem \ref{maint}, 
even when $A=Z$ is a field, the Galois $Z$-bimodule $P=Z\otimes K$ 
is not necessarily split, i.e., the eigenvalues of 
$\phi_P(z)$ don't necessarily belong to $Z$ but may define a nontrivial extension of $Z$. 
In other words, $Z$ is not necessarily a Galois extension of
$Z^K$, and the Galois group $G$ attached to $P$ does not 
have to preserve $Z$. 

Namely, let $q$ be a primitive $m$-th root of unity, 
and consider the $m^2n$-dimensional generalized Taft algebra $K=T_{m,n}$ generated by $g,x$ such that 
$g^{mn}=1$, $gx=qxg$, $x^m=g^m-1$, $\Delta(g)=g\otimes g$,
$\Delta(x)=x\otimes g+1\otimes x$
(where $n,m\ge 2$). Define the coaction of $K$ on $Z=\Bbb C(z)$ by the formula 
$\rho(z)=z\otimes g+1\otimes x$. Then we have $\rho(z^m)=z^m\otimes g^m+1\otimes (g^m-1)$, 
so $\rho(z^m+1)=(z^m+1)\otimes g^m$, and thus the invariants $Z^K$ are generated by $u=(z^m+1)^n$. 
So $[Z:Z^K]=mn$, but this is a non-Galois extension obtained by adding to $\Bbb C(u)$ 
a root of the polynomial $(t^m+1)^n-u$. The Galois closure $E$ of
this extension is the splitting field 
of this polynomial, which has degree $m^nn$ over $Z^K=\Bbb C(u)$:
$$
E=\Bbb C(u^{1/n})((-1+\varepsilon^j u^{1/n})^{1/m}, j=0,...,n-1),
$$ where
$\varepsilon$ is a primitive $n$-th root of unity. 
The Galois group $G={\rm Gal}(E/Z^K)$ is thus
of order $m^nn$, and is isomorphic to 
$\Bbb Z/n\ltimes (\Bbb Z/m)^n$, where $\Bbb Z/n$ cyclically permutes the summands. 
The eigenvalues of $\rho(z)$ include $q^j(-1+\varepsilon^i
(1+z^m))^{1/m}$, so they don't all lie in $Z$. 

This is, of course, only possible because $K$ is not basic. 
If $K$ is basic, all its irreducible representations are
1-dimensional, and thus it is clear that all the eigenvalues 
of $\rho(z)$ have to belong to $Z$.    
\end{example}

\subsection{Localization and the Galois group of a coaction}

\begin{corollary}
In the situation of Theorem \ref{maint}, the central simple
algebra $Q\otimes_Z A$ carries a coaction of $K$, which 
extends the coaction of $K$ on $A$. 
\end{corollary}

\begin{remark}
This is a special case of a much more general result, 
\cite{SV}, Theorem 2.2. 
\end{remark}

\begin{proof}
Let $L$ be the quotient field of $Z\cap A^K$. Then 
we have a coaction of $K$ on 
$L\otimes_{Z\cap A^K}A$ which extends the coaction of $K$ on
$A$. But it is easy to see that $L\otimes_{Z\cap A^K}A=Q\otimes_Z
A$. Indeed, it is well known that if $U\subset Z$ are commutative domains 
with fields of quotients $Q_U,Q_Z$ and $Z$ is integral over $U$ 
then $Q_U\otimes_UZ=Q_Z$. 

This allows us to define, in the situation of Theorem
\ref{maint}, the Galois group of the coaction.
Write $Q^K$ for $Q\cap (Q\otimes_Z A)^K$ (we don't assume that
$Q$ is $K$-costable). 

\begin{definition} The {\it Galois group} ${\rm Gal}(\rho)$ 
of the coaction $\rho$ is the group $G={\rm Gal}(E/Q^K)$, where 
$E$ is the Galois closure of $Q$ over $Q^K$.
\end{definition}

\begin{example}
1. If a finite group $\Gamma$ acts on central simple algebra 
$A$ with center $Z$ by automorphisms, then
the Galois group is $G=\Gamma/\Gamma_0$, where $\Gamma_0$ is the subgroup of elements
acting trivially on $Z$.   

2. In Example \ref{galgro}, the Galois group is 
$G=\Bbb Z/n\ltimes (\Bbb Z/m)^n$.
\end{example}
\end{proof}

\section{Galois bimodules over semisimple
  algebras finite over the center}

\subsection{Galois bimodules over commutative semisimple algebras}

Now let $L$ be a commutative semisimple algebra, i.e., a direct
sum of finitely many fields of the same characteristic: 
$L=L_1\oplus...\oplus L_n$. Then an $L$-bimodule 
$P$ has the form $P=\oplus_{i,j}P_{ij}$, where $P_{ij}$ is an
$(L_i,L_j)$-bimodule. Define the oriented incidence 
graph $\Gamma(P)$ of $P$ with vertices labeled by $1,...,n$, and 
edges $i\rightarrow j$ whenever $P_{ji}\ne 0$. 
Let us say that $P$ is connected if $\Gamma(P)$ is a connected
graph. 

For a vector ${\bold p}=(p_1,...,p_n)\in \Bbb Z_{\ge 0}^n$, 
denote by $L^{\bold p}$ the $L$-module 
$L_1^{p_1}\oplus...\oplus L_n^{p_n}$. 

\begin{definition} A {\it quasi-Galois bimodule} over $L$ of rank $d$ and type 
$(\bold p,\bold q)$
is an $L$-bimodule $P$ which is isomorphic to $L^{\bold p}$ 
as a left $L$-module, to $L^{\bold q}$ as a right $L$-module,
and such that $P\otimes_L P\cong P^d$ as a bimodule. 
\end{definition}

Clearly, a Galois $L$-bimodule of rank $d$ is 
the same thing as a quasi-Galois $L$-bimodule 
of rank $d$ and type $(d\cdot \bold 1,d\cdot \bold 1)$, where $\bold 1=(1,...,1)$. 

It is clear that any quasi-Galois $L$-bimodule of rank $d$ 
is of the form $P=\oplus_i P^{(i)}$, where 
$P^{(i)}$ is a connected quasi-Galois bimodule of rank $d$ over $L^{(i)}$, and
$L=\oplus_i L^{(i)}$ (namely, $\Gamma(P^{(i)})$ are the connected components of $\Gamma(P)$). 
Thus it suffices to consider only connected
quasi-Galois $L$-bimodules. 

Suppose ${\mathcal Z}$ is a field, and $\psi_i: {\mathcal Z}\to
L_i$ are field embeddings, such that $[L_i:{\mathcal
  Z}]=d_i<\infty$. Then ${\mathcal Z}\subset L$ by the diagonal
embedding $z\mapsto (\psi_1(z),...,\psi_n(z))$. 
Let 
\begin{equation}\label{quasigal}
P=\oplus_{i,j}(L_i\otimes_{\mathcal Z}L_j)^{a_ir_j}
\end{equation}
for some positive integers $a_i,r_j$. 
It is easy to see that $P$ is a quasi-Galois $L$-bimodule of rank $d=\sum_i a_id_ir_i$
and type $(\bold p,\bold q)$, where $p_i=a_i\sum_j r_jd_j$ and $q_j=r_j\sum_i a_id_i$.   

\begin{theorem}\label{commsem}
Let $P$ be a connected quasi-Galois $L$-bimodule. 
Then the center ${\mathcal Z}={\mathcal Z}(P)$ is a field, 
which embeds into each $L_i$, with $[L_i:{\mathcal Z}]=d_i<\infty$. 
Moreover, $P$ has the form (\ref{quasigal}). 
In particular, if $P$ is Galois, then 
it is a multiple of $L\otimes_{\mathcal Z}L$. 
\end{theorem}

\begin{proof}
Let $d$ be the rank of $P$. 
We will need the following well known lemma from linear algebra. 

\begin{lemma}\label{fp}
Let $0<p<m$ be positive integers, 
and $A\in {\rm Mat}_m(\Bbb R)$ be a 
block matrix
$$
A=\left(\begin{matrix}X& Y\\ 0&Z
\end{matrix}\right) 
$$
where $X$ is of size $p$ by $p$. 
Suppose that $A^2=A$, the entries 
of $X$ and $Z$ are strictly positive, and the entries of $Y$ are
nonnegative. Then $Y=0$.  
\end{lemma}

\begin{proof} 
For any $N\ge 2$ we have
$$
A=A^N=\left(\begin{matrix}X^N& \sum_i X^iYZ^{N-1-i}\\ 0&Z^N
\end{matrix}\right)=
$$
$$
\left(\begin{matrix}X& XY+YZ+(N-2)XYZ\\ 0&Z
\end{matrix}\right).  
$$
Thus, $XYZ=0$. Since $X,Z$ have positive entries and $Y$
nonnegative entries, we have $Y=0$, as desired.  
\end{proof} 

\begin{lemma}\label{tran}
One has $P_{ij}\ne 0$ for any $i,j$. 
\end{lemma}

\begin{proof} 
Let $[P]$ be the matrix of dimensions of $P_{ij}$ over $L_i$. 
Then $[P]^2=d[P]$.

Call two vertices of $\Gamma(P)$ 
equivalent if they can be reached from each other by moving
along the edges of $\Gamma(P)$ according to their orientation.
\footnote{Note that since $[P]^2=d[P]$, 
if $a$ vertex $v_1\in \Gamma(P)$ can be reached 
from $v_2\in \Gamma(P)$ by an oriented path then there is actually an edge from $v_2$ to $v_1$.} 
Let $C_1$ be an equivalence class from which one cannot reach any
other one (i.e., a sink), and assume for the sake of
contradiction that $C_2$ is another equivalence class from which 
there is an edge into $C_1$. Consider the principal 
submatrix $A$ of $[P]/d$ corresponding to $C_1$ and $C_2$. 
Then $A$ satisfies the conditions of Lemma \ref{fp}. 
Then the conclusion of the lemma gives a contradiction with 
the existence of an edge from $C_2$ to $C_1$. Thus, there is only
one equivalence class, which proves the
lemma (in view of the identity $[P]^2=d[P]$). 
\end{proof} 

Thus, we see that the matrix $[P]$ has rank $1$, i.e. 
$[P]_{ij}=a_ib_j$ for some positive rational numbers $a_i,b_j$. 
We may scale $a_i$ in such a way that they are positive integers
with ${\rm gcd}(a_1,...,a_n)=1$. Then $b_j$ are positive integers as well.   
Also, since $P$ is of rank $d$, we have $d=\sum_j a_jb_j$. 

Now consider the $L_i$-bimodule $P_{ii}$. 
It has left dimension $a_ib_i$, finite right dimension, and 
we have $P_{ii}\otimes_{L_i}P_{ii}\subset P_{ii}^d$.
Thus, by Proposition \ref{dimeq}, the right dimension of $P_{ii}$ is also $a_ib_i$, 
and $P_{ii}$ is a weakly Galois $L_i$-bimodule. 

Let ${\mathcal Z}_i$ be the center of $P_{ii}$. Then by 
Proposition \ref{conta3}, $[L_i:{\mathcal Z}_i]<\infty$.

Let $\phi_{ij}: L_j\to {\rm Mat}_{a_ib_j}(L_i)$ be the map defined
by the right action of $L_j$ on $P_{ij}$. 
Since $P_{ij}\otimes_{L_j}P_{ji}$ is contained in $P_{ii}^d$, 
we have 
\begin{equation}\label{centermaps}
(1\otimes \phi_{ij})(\phi_{ji}(z))={\rm Id}_{a_ib_j}\otimes {\rm
  Id}_{a_jb_i}\otimes z
\end{equation}
for any $z\in{\mathcal Z}_i$. This means that $\phi_{ji}(z)$ is a
scalar matrix, $\phi_{ji}(z)=\eta_{ji}(z){\rm Id}_{a_jb_i}$, where
$\eta_{ji}(z)\in L_j$. Moreover, since $P_{jj}\otimes_{L_j}P_{ji}$ 
is contained in $P_{ji}^d$, we have that 
$(1\otimes \phi_{jj})(\phi_{ji}(z))$ is
conjugate to restriction of ${\rm Id}_d\otimes \phi_{ji}(z)$
to an invariant subspace, which implies that $\eta_{ji}(z)\in {\mathcal Z}_j$. 
Thus, $\eta_{ji}: {\mathcal Z}_i\to {\mathcal Z}_j$, and 
by (\ref{centermaps}), $\eta_{ji}\circ \eta_{ij}={\rm Id}$, so
$\eta_{ij}$ are isomorphisms. Finally, since 
$P_{\ell j}\otimes_{L_j}P_{ji}$ is contained in $P_{\ell i}^d$, 
we have $\eta_{\ell j}\circ \eta_{ji}=\eta_{\ell i}$.

Thus, we have a single field ${\mathcal Z}$, 
and embeddings $\psi_i: {\mathcal Z}\to L_i$ such that ${\mathcal
  Z}_i=\psi_i({\mathcal Z})$ and $\eta_{ji}=\psi_j\circ
\psi_i^{-1}$. Moreover, ${\mathcal Z}={\mathcal Z}(P)$ 
(where ${\mathcal Z}$ is embedded into $L$ by the diagonal
embedding $\oplus_i \psi_i$). We will identify ${\mathcal Z}$ with ${\mathcal Z_i}$ using 
the maps $\psi_i$.  

Now we see that the dimension of $P_{ij}$ over ${\mathcal Z}$ 
(on either side) is $a_ib_jd_i$, which implies 
that the dimension of $P_{ij}$ as a right vector space over 
$L_j$ is $a_ib_jd_i/d_j$ (which is therefore an integer). 
Thus, $P$ is a quasi-Galois bimodule of
type $(\bold p,\bold q)$, where 
$$
p_i=a_i\sum_jb_j,\ q_j=\frac{b_j}{d_j}\sum_i a_id_i.
$$

Now, by Proposition \ref{conta3}, 
$P_{ii}$ contains $L_i\otimes_{\mathcal Z}L_i$. Since 
$P_{ii}\otimes_{L_i} P_{ij}$ is contained in  
$P_{ij}^d$, we get that $P_{ij}^d$ contains 
$L_i\otimes_{\mathcal Z}L_j$. Now, $L_i\otimes_{\mathcal Z}L_j$
is a Frobenius algebra, so it is an injective module over itself.
Thus, this inclusion splits and $L_i\otimes_{\mathcal Z}L_j$ 
is contained in a multiple of $P_{ij}$ as a direct summand. 
Since $L_i\otimes_{\mathcal Z}L_j$ is a commutative algebra, it
has multiplicity free decomposition into indecomposable 
projective modules, so is contained as a direct summand in
$P_{ij}$ itself. 

To conclude the proof, we use the argument 
from the proof of Proposition \ref{conta2}. 
Let $Q=\oplus_{i,j}(L_i\otimes_{\mathcal Z}L_j)^{Na_ib_j/d_j}$, 
where $N$ is the smallest positive integer such that 
the numbers $Nb_j/d_j$ are integers. 
It is easy to check that $Q$ is a quasi-Galois $L$-bimodule 
of rank $Nd$ and type $(N\bold p,N\bold q)$,
so that 
\begin{equation}\label{eq1}
Q\otimes_L Q=Q^{dN}
\end{equation}
and also 
\begin{equation}\label{eq2}
Q\otimes_L P=P\otimes_L Q=Q^d.
\end{equation}

Let $Q=\oplus_k Q_k^{s_k}$, $s_k>0$, be the decomposition of
$Q$ into indecomposable $L$-bimodules, let $t_k$ be the 
multiplicity of $Q_k$ in the decomposition of $P$, and let 
$t/s=t_\ell/s_\ell$ be the smallest of the rational numbers 
$t_k/s_k$. Since, as shown above, some multiple of $P$ contains $Q$ as
a direct summand, $t>0$. We have $P^s=Q^t\oplus M$, where $M$ does not contain
$Q_\ell$ as a direct summand. By (\ref{eq1}) and (\ref{eq2}), 
we have 
\begin{equation}\label{eq3}
Q\otimes_L M=M\otimes_L Q=Q^{d(s-tN)}.
\end{equation}
Thus, we get 
$$
Q^{tds}\oplus M^{ds}=P^{ds^2}=P^s\otimes_L P^s=
(Q^t\oplus M)\otimes_L (Q^t\oplus M)=Q^{t^2dN+2td(s-tN)}\oplus
M\otimes_L M.
$$
Since $M$ does not contain $Q_\ell$ as a direct summand, 
we get 
$$
tds\ge t^2dN+2td(s-tN). 
$$
Since $t>0$, this gives 
$$
s\ge tN+2(s-tN)=2s-tN.
$$
Thus, 
$$
s-tN\le 0.
$$ 
But we also know from (\ref{eq3}) that $s-tN\ge 0$. Thus $s-tN=0$, and hence
$M=0$. So $P^s=Q^t$,
i.e. $P=\oplus_{i,j}(L_i\otimes_{\mathcal Z} L_j)^{a_itb_jN/sd_j}$.
This means that $r_j:=\frac{tb_jN}{sd_j}$ are integers, and 
$P=\oplus_{i,j}(L_i\otimes_{\mathcal Z}L_j)^{a_ir_j}$, as claimed. 
The theorem is proved. 
\end{proof} 

\subsection{Galois bimodules over noncommutative semisimple algebras finite over center} 

Let $B$ be a semisimple algebra finite over its center, i.e. $B=\oplus_{i=1}^n B_i$,
where $B_i$ are central simple algebras over fields $L_i$ of the
same characteristic, of dimension $m_i^2$. Let $P$ be a Galois $B$-bimodule of rank $d$. 
As in the commutative case, it suffices to consider connected
bimodules, so we will assume that $P$ is connected. 

First consider the case when $B_i={\rm Mat}_{m_i}(L_i)$. 
In this case, similarly to the case of simple algebras, we have
the following proposition. Let $m_*$ be the greatest common
divisor of the $m_i$.  

\begin{proposition}\label{matr1} 
The center ${\mathcal Z}$ of $P$ is a field
such that $[L_i:{\mathcal Z}]<\infty$, and 
$P$ is a multiple of $\oplus_{i,j}{\rm Mat}_{m_i\times
  m_j}(L_i\otimes_{\mathcal Z}L_j)^{m_im_j/m_*^2}$. 
\end{proposition}

\begin{proof}
Let $L=\oplus_{i=1}^n L_i$ be the center of $B$. 
Let $F: B-{\rm Bimod}\to L-{\rm bimod}$ 
be the standard Morita equivalence. 
It is easy to see that $F(P)$ 
is a quasi-Galois $L$-bimodule 
of rank $d$. Thus, by Theorem \ref{commsem},
$F(P)=\oplus_{i,j}(L_i\otimes_{\mathcal Z}L_j)^{a_ir_j}$, where $a_i$ are coprime integers.
By looking at the left dimensions of $P_{ij}$ over $L_i$, we get
$$
\sum_j m_im_ja_ir_jd_j=dm_i^2.
$$
where $d_j=[L_j:{\mathcal Z}]$. 
This implies that $a_i=m_i/m_*$, and 
$dm_*=\sum_j m_jr_jd_j$. 
Also, by looking at right dimensions we have 
$$
\sum_i m_im_ja_ir_jd_i=dm_j^2,
$$
which implies that 
$r_j=rm_j/m_*$, where 
$$
r=\frac{dm_*}{\sum_i d_im_ia_i}=\frac{dm_*^2}{\sum_i d_im_i^2}.
$$ 
is an integer. 
This implies the statement by applying $F^{-1}$.   
\end{proof}

Now consider the general case. 

\begin{proposition}\label{censim2} 
The center ${\mathcal Z}$ of $P$ is a field such that
$[L_i:{\mathcal Z}]<\infty$, and 
$P^{m_*^2}$ is a multiple of $B\otimes_{\mathcal Z}B$. 
\end{proposition}

\begin{proof} 
Let $E_i\supset L_i$ be finite
dimensional commutative separable algebras (not necessarily field extensions) 
which split the central simple algebras $B_i$, i.e. 
$E_i\otimes_{L_i}B_i={\rm Mat}_{m_i}(E_i)$. We can 
pick $E_i$ in such a way that $[E_i:L_i]=D$, the same number for
all $i$. Let $E=\oplus_i E_i$. Then $E=L^D$ as an $L$-module, 
so by Lemma \ref{ext}(i), $E\otimes_L P\otimes_L E$
is a Galois bimodule over $E\otimes_L B=\oplus_i {\rm
  Mat}_{m_i}(E_i)$. By Proposition \ref{matr1}, 
$E\otimes_L P\otimes_L E$ is a multiple of 
$\oplus_{i,j}{\rm Mat}_{m_i\times
  m_j}(E_i\otimes_{\mathcal Z}E_j)^{m_im_j/m_*^2}$, where 
${\mathcal Z}$ is the center of $P$. 
This implies that the $E$-bimodule $(E\otimes_L P\otimes_L E)^{m_*^2}$
is a multiple of 
$$
\oplus_{i,j}{\rm Mat}_{m_i\times m_j}(E_i\otimes_{\mathcal
  Z}E_j)^{m_im_j}=\oplus_{i,j}{\rm
  Mat}_{m_i}(E_i)\otimes_{\mathcal Z}{\rm Mat}_{m_j}(E_j)=(E\otimes_L B)\otimes_{\mathcal Z}(E\otimes_L B). 
$$
So restricting back to $B$,
we get that $P^{D^2m_*^2}$ is a multiple of 
$(B\otimes_{\mathcal Z}B)^{D^2}$, which shows that 
$P^{m_*^2}$ is a multiple of $B\otimes_{\mathcal Z}B$, as desired. 
\end{proof} 

\begin{corollary}\label{charpoly} 
Let $P$ be a Galois bimodule over
$B$. Let $m$ be the least common multiple of the $m_i$. 
Given $a\in (a_1,...,a_n)\in L\subset B$,
denote by $\phi_i(a)$ the $i$-th component 
of $\phi(a)=\phi_P(a)$, $\phi_i(a)\in {\rm Mat}_{dm_i^2}(L_i)$. 
Let $\chi_a\in L[t]$ be the collection of monic polynomials 
$\chi_{\phi_i(a)}(t)^{m^2/m_i^2}\in L_i[t]$ of degree $dm^2$, 
where $\chi_{\phi_i(a)}$ is the characteristic polynomial 
of $\phi_i(a)$. Then the coefficients of $\chi_a$ 
belong to the center ${\mathcal Z}={\mathcal Z}(P)\subset L$. 
\end{corollary}  

\begin{proof}
It suffices to assume that $P$ is connected, so we restrict ourselves 
to this case. Let $\chi_{a_j}$ be the characteristic polynomial of $a_i\in L_i$
acting by multiplication on $L_i$ as a vector space over
${\mathcal Z}$. Clearly, this is a polynomial over ${\mathcal
  Z}$. By Proposition \ref{censim2}, we have 
$$
\chi_{\phi_i(a)}^{m_*^2}=\prod_{j=1}^n\chi_{a_j}^{m_i^2m_j^2}=f^{m_i^2},
$$
where $f=\prod_{j=1}^n\chi_{a_j}^{m_j^2}$. 
Thus, 
$$
\chi_{\phi_i(a)}^{m^2/m_i^2}=f^{(m/m_*)^2}.
$$
This is a polynomial with coefficients in ${\mathcal Z}$, and it
is independent on $i$, which implies that $\chi_a$ has
coefficients in ${\mathcal Z}$, as desired. 
\end{proof}

\begin{corollary}\label{divisi1}
If $P$ is a connected Galois $B$-bimodule of rank $d$, and
$d_i=[L_i:{\mathcal Z}]$, then $\sum_{i=1}^nd_i(m_i/m_*)^2$
divides $d$. 
\end{corollary}

\begin{proof} 
By Proposition \ref{censim2}, $P^{m_*^2}=(B\otimes_{\mathcal
  Z}B)^r$, so computing ranks over $B$ as left modules, 
we get $dm_*^2=r\sum d_im_i^2$, which implies the statement. 
\end{proof}

\section{Generalization of Theorem \ref{maint}}

In this section we use the results of the
previous section to generalize Theorem \ref{maint} 
to the situation when $Z$ is not necessarily a domain, 
but a reduced algebra, i.e., one without nonzero nilpotents.
Let $Q$ be the total quotient ring of $Z$.   

Let $A$ be an algebra over $Z$ with a coaction of a finite
dimensional Hopf algebra $K$ over $k$ of dimension $d$. 
As in Section 3, consider assumptions (1)-(5) on 
$A,Z$, generalizing assumption (2) as follows: 

(2) $Q\otimes_Z A$ is a semisimple algebra with center $Q$. 

Note that (2) implies that $Q=Q_1\oplus...\oplus
Q_n$, where $Q_i$ are fields containing $k$. 
This property of $Z$ is satisfied, for instance, 
if $Z$ is Noetherian (in particular, affine).

\begin{theorem}\label{maint1}
Theorem \ref{maint} holds in this more general situation. 
In other words, under the conditions (1)-(4) or (1),(2),(5), 
$A$ and $Z$ are integral over $Z^K$, and if $Z$ is finitely generated over $k$, then so is $Z^K$, 
and $Z$ is a finite module over $Z^K$.  
\end{theorem}

In the case when $A$ is commutative, this is again a special case 
of the result of \cite{S1} (for reduced agebras whose quotient rings 
are finite direct sums of fields). 
 
\begin{proof}
The proof is parallel to the proof of Theorem \ref{maint}. 
The proofs of (iii) and (iv) are the same, so we only comment on
the proofs of (i) and (ii). 

As in the proof of Theorem \ref{maint}, let $P=A\otimes K$, and
$P_{\rm loc}=Q\otimes_Z P$. 
This is a left module over the algebra
$B:=Q\otimes_Z A$, which is a semisimple algebra 
over $Q$ by (2), i.e., a direct sum of central simple algebras 
$B_i$ of dimensions $m_i^2$ over $Q_i$ for some $m_i$. 
Consider the right action of $Z$ 
on $Q\otimes_Z P$. This action defines a homomorphism 
$\psi_P: Z\to {\rm End}_Q(Q\otimes_Z P)\cong \oplus_{i=1}^n{\rm
  Mat}_{dm_i^2}(Q_i)$. 

We claim that for any $z\in Z$ which is not a zero divisor, the element
$\psi_P(z)$ is invertible. 
Indeed, let $x\in Q\otimes_Z P$ be such that $\psi_P(z)x=xz=0$. 
Let $x=w^{-1}y$, where $w\in Z$ (not a zero divisor), and $y\in P$. 
Hence, $yz$ is a 
torsion element of $P$, i.e. 
$z'yz=0$ in $P$ for some
$z'\in Z$, which is not a zero divisor. But by (3) (or by (5)), 
$A$ is torsion-free over $Z$, which implies 
that $P$ is torsion-free over $Z$ on each side.
So we get that $yz=0$ and hence $y=0$ and $x=0$, 
as desired. 

Thus, we see that the right action of $Z$ on 
$Q\otimes_Z P$ naturally extends to a right action of $Q$, i.e. 
$Q\otimes_Z P$ is a $Q$-bimodule. Hence, $Q\otimes_Z P$ is a 
bimodule over the semisimple algebra $B$, and 
$$
P_{\rm loc}=Q\otimes_Z P=Q\otimes_Z P\otimes_Z Q=B\otimes_A P\otimes_A B.
$$  

By Proposition \ref{coa}, $P:=A\otimes K$
is a Galois bimodule over $A$ of rank $d$.
This implies by Lemma \ref{ext}(i) 
that $P_{\rm loc}$ 
is a Galois bimodule over $B$ of rank $d$.

By Corollary \ref{charpoly}, 
for each $a\in Z$, the polynomial 
$\chi_a(t)\in Q[t]$ defined in Corollary \ref{charpoly}
has coefficients in ${\mathcal Z}(P_{\rm loc})=Q^K=Q\cap A^K$. 

Now, since $Z$ is integrally closed (i.e., is a normal
algebra), it is a direct sum of integrally closed 
domains $Z_1\oplus...\oplus Z_n$. Hence, 
Lemma \ref{inclo} implies that $\chi_a(t)\in Z[t]$. 

Moreover, if $z\in Z$ is central for $P_{\rm loc}$ 
then by (3) it is also central for $P$. 
Thus, the coefficients of $\chi_a$ belong 
to $Z\cap {\mathcal Z}(P)$. 
By Proposition \ref{coa}, this means that 
these coefficients belong to $Z\cap A^K$.
Hence $Z$ is integral over $Z\cap A^K$
(as $a\in Z$ is annihilated by the polynomial $\chi_a$). 
Thus, (i) is proved. 

To prove (ii), it suffices to note that by (5), 
for $a\in Z$ the coefficients of the polynomial $\chi_a(t)$ of
Proposition \ref{censim2} are in $Z$, 
after which the proof is the same as in (i). 
\end{proof} 

\begin{corollary}\label{loc} 
In the situation of Theorem \ref{maint1}, 
the coaction of $K$ on $A$ uniquely extends 
to $Q\otimes_Z A$. 
\end{corollary}

\begin{proof}
This follows since by Theorem \ref{maint1}, 
$Q\otimes_Z A=Q^K\otimes_{Z^K}A$. 
\end{proof}

\begin{remark}
As before, this is a special case of  
\cite{SV}, Theorem 2.2. 
\end{remark}

We also obtain the following generalization of Proposition \ref{divisi}:

\begin{proposition}\label{divisi2} 
Suppose $A=\oplus_{i=1}^nA_i$ is an indecomposable semisimple 
$K$-comodule algebra with center $Q$,
where $A_i$ are simple algebras of degrees $m_i$ over fields
$Q_i$ (so that $Q=\oplus_{i=1}^n Q_i$). 
Let $m_*$ be the greatest common divisor
of the $m_i$, and $d_i=[Q_i:Q^K]$. 
Then $\sum_{i=1}^n d_i(m_i/m_*)^2$ divides
$d=\dim K$. 
\end{proposition}

\begin{proof} $A\otimes K$ is a Galois $A$-bimodule of rank $d$,
  so the statement follows from Proposition \ref{divisi1}. 
\end{proof} 

\begin{remark}
Here is another proof of Proposition \ref{divisi2} in characteristic zero,
using the theory of tensor categories (this approach also works in characteristic $p$ with some complications). 
By Theorem \ref{maint1}, $[L_i:{\mathcal Z}]=d_i<\infty$. So, tensoring over ${\mathcal Z}$ 
with the algebraic closure $\overline{\mathcal Z}$, we obtain a semisimple algebra $\overline{A}$ 
with a coaction of a Hopf algebra $\overline{K}$ (both finite dimensional over $\overline{\mathcal Z}$). 
Then ${\rm Rep}\overline{A}$ is a semisimple indecomposable module category 
over ${\rm Rep}\overline{K}$. Thus, one can define canonical Frobenius-Perron dimensions
of objects in ${\rm Rep}\overline{A}$, as in \cite{ENO}, Subsection 2.5. They are defined 
by the formula 
$$
{\rm FPdim}(M_i)^2=\frac{dn_i^2}{\sum_j n_j^2},
$$
where $n_j$ is the ordinary dimension of $M_j$. 
These dimensions are known to be integers, since 
${\rm FPdim}(M_i)^2={\rm FPdim} \underline{\rm End}(M_i)$. 
This implies that if $n_*$ is the greatest common divisor of the 
$n_i$ then $\frac{dn_*^2}{\sum_j n_j^2}\in \Bbb Z$. 
But $n_*=m_*$, and $(n_j)$ is the collection where each $m_i$ is repeated 
$d_i$ times. This implies the statement.   
\end{remark}

\section{The tensor category of finite dimensional $L$-bimodules}

The theory of Galois bimodules over a field $L$ is closely
related to the theory developed in \cite{FO}. 
Let us discuss this connection. 
For simplicity, we assume that ${\rm char}(L)=0$, although 
this discussion can be extended to positive characteristic as
well (at the cost of losing semisimplicity of bimodules). 

In \cite{FO}, the authors consider 
the semisimple tensor category $L-{\rm bimod}_E$ of 
finite dimensional $L$-bimodules 
$M$ that split over some finite Galois extension 
$E$ of $L$ (or are $E$-balanced, in the language of \cite{FO}), 
i.e., $E\otimes_LM\otimes_L E$ 
is a direct sum of bimodules $Eg$ for $g\in {\rm Aut}(E)$. 
They show that the Grothendieck ring of this category tensored
with $\Bbb Q$, $K_0(L-{\rm bimod}_E)\otimes{\Bbb Q}$,
is naturally isomorphic to the Hecke algebra ${\mathcal H}({\rm
  Aut}(E),H)$, where $H={\rm Gal}(E/L)$ 
(i.e., the convolution algebra of $\Bbb Q$-valued 
$H$-biinvariant functions on ${\rm Aut}(E)$ with finite support). 
Thus, a Galois $L$-bimodule $P$ of rank $d$ which splits over $E$ 
defines a nonnegative idempotent $e=\frac{1}{d}[P]$ 
in this Hecke algebra. The support of such an idempotent is a 
multiplicatively closed finite subset of ${\rm Aut}(E)$,
so it is a finite subgroup $G\subset {\rm Aut}(E)$ containing $H$.   
Moreover, it is easy to see as in the proof of Proposition
\ref{p1} that the idempotent $e$ has to be 
defined by the formula $e(g)=\frac{1}{|G|}$ for all $g\in G$. 
So $[P](g)=\frac{d}{|G|}$. For this to define an integral class, 
we need $r:=|H|d/|G|$ to be an integer, and 
then $E\otimes_L P\otimes_L E=P(E,G)^r$, as explained
in Proposition \ref{spli2}. 

The results of Section 2 on Galois bimodules 
and the results of \cite{FO} can be used to prove the
following result about the tensor category $L-{\rm bimod}$
of $L$-bimodules which are finite dimensional as left and right $L$-vector 
spaces, which also gives a classification of Galois
$L$-bimodules.   

\begin{theorem}\label{tenscat} 
Suppose ${\mathcal C}\subset L-{\rm bimod}$ 
is a full abelian rigid tensor subcategory, which is semisimple and has
finitely many simple objects. Then there exists a unique 
subfield $F\subset L$ such that $[L:F]<\infty$, 
and ${\mathcal C}={\mathcal C}(F,L)$ 
is the category of $L$-bimodules which are linear over 
$F$, i.e. the right and left actions of $F$ coincide. 
The simple objects of this category are the simple 
direct summands in $L\otimes_F L$,
and they split over the Galois closure $E$ of $L$ over $F$.  
\end{theorem}

\begin{proof}
Let $d(X)$ be the dimension of an $L$-bimodule $X$ as a left $L$-vector space.
Then $d: K_0({\mathcal C})\to \Bbb Z$ is a character. 
Let $X_i$ be the simple objects of ${\mathcal C}$. 
Let $X=\oplus_i X_i$.  
Since ${\mathcal C}$ is a tensor category, 
the matrix of right multiplication by $X$ in the basis $X_i$ 
 is an indecomposable nonnegative matrix. So by the Frobenius-Perron theorem there exist unique up to scaling 
positive real numbers $r_i$ such that $(\sum_i r_iX_i)X=\Lambda \sum_i r_iX_i$ 
in $K_0({\mathcal C})\otimes \Bbb R$, where $\Lambda$ is the largest positive eigenvalue of $X$. 
Computing the dimensions of both sides as a left vector space, we get 
that $\Lambda=d(X)$, which is an integer. This implies that 
the numbers $r_i$ can be scaled to be all integers as well; 
let us choose them in such a way. Let $P=\oplus_i r_i X_i$. 
By uniqueness of $r_i$, we get that $P\otimes_L P=P^{d(P)}$. 
Thus, $P$ is a Galois bimodule. So by Theorem \ref{classif}, 
$P$ is a multiple of $L\otimes_F L$ for some $F$ (namely,
$F={\mathcal Z}(P)$). This implies the desired
statement (as any simple $F$-linear $L$-bimodule is a quotient,
hence a direct summand, of $L\otimes_F L$).   
\end{proof} 

\begin{remark}
In the tensor category ${\mathcal C}(F,L)$, one has 
${\rm End}(\bold 1)=L$, but in general it is not 
$L$-linear but only $F$-linear.
It is a form over $F$, in the sense of \cite{EG}, 
of the multifusion category of ${\rm Fun}(G/H,E)$-bimodules, 
which is associated to the natural semilinear action 
of $G={\rm Gal}(E/F)$ on this category. The simple objects 
in this $F$-linear tensor category are then associated 
to orbits of $G$ on $G/H\times G/H$, 
which correspond to double cosets of $H$ in $G$, spanning the
Hecke algebra as in \cite{FO}. Note also that this category 
is independent on the choice of the Galois extension $E$ of 
$F$, as long as it contains $L$. 
\end{remark}

\begin{remark}
Theorem \ref{tenscat} generalizes to the case when the field 
$L$ is replaced by a central simple algebra $B$ 
with center $L$. Namely, recall that by 
Proposition \ref{censim1},
any Galois $B$-bimodule is semisimple and  
is a rational multiple of the bimodule $B\otimes_F B$ 
for a uniquely determined 
subfield $F\subset L$ such that $[L:F]<\infty$
(namely, $F={\mathcal Z}(P)$).
Now, if ${\mathcal C}\subset B-{\rm bimod}$ 
is a full abelian rigid tensor subcategory, which is semisimple and has
finitely many simple objects, then one shows similarly 
to Theorem \ref{tenscat} that there exists a unique 
subfield $F\subset L$ such that $[L:F]<\infty$, 
and ${\mathcal C}={\mathcal C}(F,B)$ 
is the category of $B$-bimodules which are linear over 
$F$, i.e. the right and left actions of $F$ coincide. 
The simple objects of this category are the simple direct 
summands in $B\otimes_F B$.

This is related to the results of \cite{EG} in the following way.  
The category ${\mathcal C}(F,B)$ is a twisted form of the category ${\mathcal C}(F,L)$ 
in the sense of \cite{EG}. According to \cite{EG}, such twisted forms 
which split over the Galois extension $E$ of $F$ containing $L$
correspond to elements of $H^2(G,{\rm Aut}_{\otimes}({\rm Id}))$.
In our situation ${\rm Aut}_{\otimes}({\rm Id})={\rm Fun}(G/H,E^\times)/E^\times$, 
and the long exact sequence of cohomology and the Shapiro lemma imply that 
$H^2(G,{\rm Aut}_{\otimes}({\rm Id}))$ contains the quotient 
$H^2(H,E^\times)/{\rm Im}H^2(G,E^\times)$, which parametrizes 
$E$-split central simple algebras $B$ over $L$ up to Morita equivalence 
modulo those of the form $B=L\otimes_F A$, where $A$ is a central simple algebra
over $F$. It is easy to show that under the correspondence of \cite{EG},
the category ${\mathcal C}(F,B)$ corresponds precisely to the class 
of the central simple algebra $B$ in $H^2(H,E^\times)/{\rm Im}H^2(G,E^\times)$.

In fact, it is easy to see explicitly that the categories ${\mathcal C}(F,B_1)$ 
and ${\mathcal C}(F,B_2)$ are equivalent if $B_1=B_2\otimes_F A$, where $A$ is a central simple algebra over $F$. 
The equivalence is defined by the formula $M\mapsto M\otimes_F A$ for a $B_2$-bimodule $M$ linear over $F$.  
\end{remark}

\end{document}